\documentclass[12pt,reqno]{amsart}
\usepackage[T1]{fontenc}
\usepackage{amsfonts}
\usepackage{amssymb}
\usepackage{mathrsfs}
\usepackage[latin1]{inputenc}
\usepackage{amsmath}
\usepackage[english]{babel}
\usepackage{setspace}

 \newtheorem{thm}{Theorem}[section]
 \newtheorem{cor}[thm]{Corollary}
 \newtheorem{lem}[thm]{Lemma}
 \newtheorem{prop}[thm]{Proposition}
 \theoremstyle{definition}
 \newtheorem{defn}[thm]{Definition}
 \newtheorem{rem}[thm]{Remark}
 
 \numberwithin{equation}{section}

 \numberwithin{equation}{section}

\newcommand{\R}{{\mathbb R}}

\newcommand{\Z}{{\mathbb Z}}
\newcommand{\C}{{\mathbb C}}
\newcommand{\N}{{\mathbb N}}

\newcommand{\cD}{{\mathcal D}}
\newcommand{\cL}{{\mathcal L}}
\newcommand{\cF}{{\mathcal F}}
\newcommand{\cE}{{\mathcal E}}

\newcommand{\cC}{{\mathcal C}}
\newcommand{\cB}{{\mathcal B}}

\newcommand{\cM}{{\mathcal M}}

\newcommand{\cS}{{\mathcal S}}
\newcommand{\cO}{{\mathcal O}}

\newcommand{\su}{\subseteq}

\newcommand\proj{\mathop{\rm proj\,}}
\newcommand\ind{\mathop{\rm ind\,}}

\begin{document}

%
%
%
%

\title[Convolutors on $\cS_{\omega}(\R^N)$ ]
 {Convolutors on $\cS_{\omega}(\R^N)$}

\author[A.A. Albanese, C. Mele]{Angela\,A. Albanese and Claudio Mele}

\address{ Angela A. Albanese\\
Dipartimento di Matematica e Fisica ``E. De Giorgi''\\
Universit\`a del Salento- C.P.193\\
I-73100 Lecce, Italy}
\email{angela.albanese@unisalento.it}

\address{Claudio Mele\\
	Dipartimento di Matematica e Fisica ``E. De Giorgi''\\
	Universit\`a del Salento- C.P.193\\
	I-73100 Lecce, Italy}
\email{claudio.mele1@unisalento.it}





\begin{abstract}
In this paper we  continue the study of the 
spaces $\cO_{M,\omega}(\R^N)$ and $\cO_{C,\omega}(\R^N)$ undertaken in \cite{AC}. We determine new representations  of such spaces and we give some structure theorems for their dual spaces. Furthermore, we show that $\cO'_{C,\omega}(\R^N)$ is the space of convolutors of the  space $\cS_\omega(\R^N)$ of the $\omega$-ultradifferentiable rapidly decreasing functions of Beurling type (in the sense of Braun, Meise and Taylor) and of its dual space $\cS'_\omega(\R^N)$. We also establish  that the Fourier transform is 
 an isomorphism from $\cO'_{C,\omega}(\R^N)$
  onto $\cO_{M,\omega}(\R^N)$. In particular, we prove that this  isomorphism is  topological when the former space is endowed with the strong operator lc-topology induced by $\cL_b(\cS_\omega(\R^N))$ and the last space is endowed with its natural lc-topology.
\end{abstract}

\maketitle
\section{Introduction }\label{intro}

The study of the space $\cO_M(\R^N)$ of multipliers and  of the space $\cO_C(\R^N)$ of convolutors of the space
$\cS(\R^N)$ of rapidly decreasing functions was started by Schwartz \cite{S}. Since then, the
spaces $\cO_M(\R^N)$ and  $\cO_C(\R^N)$ attracted the attention of several authors, even recently, (see, f.i.,
\cite{BO,Gro,H,Ki,Ki1,Ku,Ku1,La,LaW} and the references therein). Their interest lies in the rich topological structure and in the importance of their application to the study of partial differential equations. In the case of ultradifferentiable classes of rapidly decreasing functions of Beurling or Roumieu type in the sense of Komatsu \cite{K} or in the frame of Gelfand-Shilov spaces,  the spaces of multipliers and of convolutors have been also considered and studied in the recent years (see, f.i., \cite{De2,DPV,DPPV,Kov,Kov-2,Pi,Pi2,Z,Z2}). Inspired by  this
line of research, in \cite{AC} the authors recently   introduced and studied the
space $\cO_{M,\omega}(\R^N)$ of the slowly increasing functions of Beurling type in the setting of
ultradifferentiable function space in the sense of \cite{BMT}, showing  there that it is the space of  multipliers of the space $\cS_\omega(\R^N)$ of  the $\omega$-ultradifferentiable rapidly decreasing functions of Beurling type, as introduced by
Björck \cite{B}. An analogous result for a more general classes of ultradifferentiable rapidly decreasing functions of Beurling type
or Roumieu type was  recently established also in \cite{De}.
 We point out that, in
general, the ultradifferentiable  classes defined in one way cannot be defined in the other way (see \cite{BMM}).

In this paper we  continue the study of the 
spaces $\cO_{M,\omega}(\R^N)$ and $\cO_{C,\omega}(\R^N)$ undertaken in \cite{AC}, where their elements were defined in terms of weighted $L^\infty$-norms. Our main aim is to show that  $\cO'_{C,\omega}(\R^N)$ is the space of convolutors of the  space $\cS_\omega(\R^N)$  and of its dual space $\cS'_\omega(\R^N)$ (see Section  5). To this end, Sections  3 and  4 are devoted to  establish all the necessary results. In particular, in Section  3 we  first prove that the elements of both spaces $\cO_{M,\omega}(\R^N)$ and $\cO_{C,\omega}(\R^N)$  can be also defined in terms of weighted $L^p$-norms and then  we give some structure theorems for their dual spaces. The characterization in terms of $L^p$-norms relies on an appropriate weighted Sobolev embedding theorem (Proposition \ref{PropMorrey}).  In Section  4 we introduce the weighted Fr\'echet spaces $\cD_{L^p_\mu,\omega}(\R^N)$, $1\leq p\leq\infty$, and study the convolution   operators on their duals, i.e., on the weighted spaces of ultradistributions $\cD'_{L^p_\mu,\omega}(\R^N)$ (of
Beurling type) of $L^{p'}$-growth, $p'$ being the conjugate exponent of $p$. Finally, in Section  6
we  establish that the Fourier transform is 
an isomorphism from $\cO'_{C,\omega}(\R^N)$
onto $\cO_{M,\omega}(\R^N)$. This isomorphim   is topological when the former space is endowed with the strong operator lc-topology induced by $\cL_b(\cS_\omega(\R^N))$ and the last space is endowed with its natural lc-topology.
 We point out that the methods of the proofs   are different from the ones used in \cite{De,De2}. 
Indeed,  in \cite{De,De2}  the proofs relies on  tools  from the time-frequency analysis as the short time Fourier transform (STFT).

\section{Preliminary}

We first give the definition of non-quasianalytic weight function in the sense of Braun-Meise-Taylor \cite{BMT} suitable
for the Beurling case, i.e., we also consider the logarithm as a weight function.

\begin{defn}\label{D.weight} A non-quasianalytic weight function is a continuous increasing function $\omega:[0,\infty)\to[0,\infty)$  satisfying the following properties:
	\begin{itemize}
		\item[($\alpha$)] there exists $K\geq1$ such that $\omega(2t)\leq K(1+\omega(t))$ for every $t\geq0$;
		\item[($\beta$)] $\int_{1}^{\infty} \frac{\omega(t)}{1+t^2}\, dt < \infty$;
		\item[($\gamma$)] there exist $a \in \mathbb{R}$, $b>0$ such that $\omega(t) \geq a+b\log(1+t)$, for every $t\geq0$;
		\item[($\delta$)] $\varphi_\omega(t)= \omega \circ \exp(t)$ is a convex function.
	\end{itemize}
\end{defn}


\begin{rem}\label{wheight prop gamma'}\rm 
	We recall some known properties of the weights functions that shall be useful in the following; the proofs can be found in the literature. 
	
	Let $\omega$ be a non-quasianalytic weight function. Then the following properties are satisfied.
	
	(1) Condition $(\alpha)$ implies for every $t_1, t_2 \geq 0$ that
	\begin{equation}\label{sub}
	\omega(t_1+t_2)\leq K(1+\omega(t_1)+\omega(t_2)).
	\end{equation}
	Observe that this condition is weaker than subadditivity (i.e.,  $\omega(t_1+t_2)\leq \omega(t_1) + \omega(t_2))$. The weight functions satisfying ($\alpha$) are not necessarily subadditive in general.
	
	(2) Condition $(\alpha)$ implies that there exists $L\geq 1$ such that for every $t\geq 0$
	\begin{equation}\label{l}
	\omega(e t)\leq L(1+\omega(t)).
	\end{equation}
	
	
	(3) By condition $(\gamma)$ we have  for every  $\lambda\geq \frac{N+1}{bp}$ that
	\begin{equation}\label{eq.Lpspazi}
	e^{-\lambda \omega(t)}\in L^p(\R^N).
	\end{equation}
\end{rem}

Given a non-quasianalytic weight function $\omega$, we define the Young conjugate $\varphi^*_\omega$ of $\varphi_\omega$ as the function $\varphi^*_\omega:[0,\infty)\to [0,\infty)$ by
\begin{equation}\label{Yconj}
\varphi^*_\omega(s):=\sup_{t\geq 0}\{ st-\varphi_\omega(t)\},\quad s\geq 0.
\end{equation}
There is no loss of generality to assume that $\omega$ vanishes on $[0,1]$. Moreover $\varphi^*_\omega$ is convex and increasing, $\varphi^*_\omega(0)=0$ and $(\varphi^*_\omega)^*=\varphi_\omega$. Further useful properties of $\varphi^*_\omega$  are listed in the following lemma, see \cite{BMT}.

\begin{lem}\label{L.Pfistar} Let $\omega\colon [0,\infty)\to [0,\infty)$ be a non-quasianalytic weight function. Then the following properties are satisfied.
	\begin{enumerate}
		\item[\rm (1)] $\frac{\varphi^*_\omega(t)}{t}$ is an increasing function in $(0,\infty)$.
		\item[\rm (2)] For every $s,t\geq 0$ and $\lambda>0$
		\begin{equation}\label{secondprop}
		2\lambda\varphi^*_\omega\left(\frac{s+t}{2\lambda}\right)\leq \lambda \varphi^*_\omega\left(\frac {s}{\lambda}\right)+\lambda\varphi^*_\omega\left(\frac{t}{\lambda}\right) \leq \lambda\varphi^*_\omega\left(\frac{s+t}{\lambda}\right).
		\end{equation}
		\item[\rm (3)] For every $t\geq 0$ and $\lambda>0$
		\begin{equation}\label{eq.bb}
		\lambda L \varphi^*_\omega\left(\frac{t}{\lambda L}\right)+t\leq \lambda \varphi^*_\omega\left(\frac{t}{\lambda}\right)+\lambda L,
		\end{equation}
		where $L\geq 1$ is the costant appearing in formula $(\ref{l})$.
		\item[\rm (4)] For every $m,M\in\mathbb{N}$ with $M\geq mL$, where $L$ is the constant appearing in formula $(\ref{l})$, and for every $t\geq 0$
		\begin{equation}\label{firstprop}
		2^{t}\exp\left(M\varphi^*_\omega\left(\frac{t}{M}\right)\right)\leq C \exp\left(m\varphi^*_\omega\left(\frac{t}{m}\right)\right),
		\end{equation}
		with $C:=e^{mL}$.
		\end{enumerate}
\end{lem}

We now introduce the ultradifferentiable function spaces and their duals of Beurling type in the sense of Braun, Meise and Taylor \cite{BMT}. 

\begin{defn}\label{D.Beurling} Let $\omega$ be a non-quasianalytic weight.
	
	(a) For a compact subset $K$ of $\R^N$ and $\lambda >0$ define
	\[
	\cE_{\omega,\lambda}(K):=\left\{f\in C^\infty(K)\colon p_{K,\lambda}(f):=\sup_{x\in K}\sup_{\alpha\in\N_0^N}|\partial^{\alpha}f(x)|\exp\left(-\lambda \varphi^*_\omega\left(\frac{|\alpha|}{\lambda}\right)\right)<\infty\right\}.
	\]
	Then $(\cE_{\omega,\lambda}(K), p_{K,\lambda})$ is a Banach space.
	
	(b) For an open subset $\Omega$ of $\R^N$ define
	\[
	\cE_\omega(\Omega):=\left\{f\in C^\infty(\Omega)\colon p_{K,m}(f)<\infty\  \forall K \Subset \Omega,\, m\in\N  \right\}
	\]
	and endow it with its natural Fr\'echet space topology, i.e., with  the lc-topology generated by the system of seminorms $\{p_{K,m}\}_{K\Subset \Omega, m\in\N}$.
	The elements of $\cE_\omega(\Omega)$ are called \textit{$\omega$-ultradifferentiable functions
		of Beurling  type} on $\Omega$. The dual $\cE'_\omega(\Omega)$ of $\cE_\omega(\Omega)$ is endowed with its strong topology. 
	
	(c) For a compact subset $K$ of $\R^N$ define
	\[
	\cD_\omega(K):=\left\{f\in \cE_\omega(\R^N)\colon {\rm supp}\, f\su  K\right\}
	\]
	and endow it with the Fr\'echet space topology generated by the sequence $\{p_{K, m}\}_{m\in\N}$ of norms. For an open subset $\Omega$ of $\R^N$ define
	\[
	\cD_\omega(\Omega):=\ind_{j\rightarrow}\cD_{\omega}(K_j),
	\]
	where $\{K_j\}_{j\in\N}$ is any fundamental sequence of compact subsets of $\Omega$. The elements of $\cD_\omega(\Omega)$ are called \textit{test functions of Beurling type} on $\Omega$. The dual $\cD'_\omega(\Omega)$ of $\cD_\omega(\Omega)$ is endowed with its strong topology. The elements
	of $\cD'_\omega(\Omega)$ are called \textit{$\omega$-ultradistributions of Beurling  type} on $\Omega$.
	
	(d) We denote by  $\mathcal{S}_\omega(\mathbb{R}^N)$  the set of all functions $f\in L^1(\mathbb{R}^N)$ such that $f,\hat{f}\in C^\infty(\mathbb{R}^N)$ and for each $\lambda >0$ and $\alpha\in\N_0^N$ we have
	\begin{equation}\label{cond Sw 1}
	\| \exp(\lambda\omega)\partial^\alpha f\|_\infty <\infty\ \ {\rm and }\ \  
	\|\exp(\lambda\omega)\partial^\alpha\hat{f}\|_\infty <\infty \; ,
	\end{equation}
	where $\hat{f}$ denotes the Fourier transform of $f$. The elements of  $\mathcal{S}_\omega(\mathbb{R}^N)$ are called \textit{$\omega$-ultradifferentiable rapidly decreasing functions of Beurling type}. 
	We denote by  $\cS'_{\omega}(\R^N)$ the dual of $\cS_{\omega}(\R^N)$  endowed with its strong topology.
\end{defn}
We refer to \cite{B,BMT} for the main properties of the spaces $\cE_\omega(\Omega)$, $\cD_\omega(\Omega)$ and $\cS_\omega(\R^N)$. In particular, we recall what follows.
	
	
	
	
	


\begin{rem}\label{in s} Let $\omega$ be a non-quasianalytic  weight function. 
	
	(1)	The condition $(\gamma)$ of Definition \ref{D.weight} implies that  $\mathcal{S}_\omega(\mathbb{R}^N)\su \mathcal{S}(\mathbb{R}^N)$  with continuous inclusion. 
	Accordingly,  we can rewrite the definition of $\mathcal{S}_\omega(\mathbb{R}^N)$ as the set of all the rapidly decreasing functions that satisfy the condition  $(\ref{cond Sw 1})$. 
	
		(2)	The space $\mathcal{S}_\omega(\mathbb{R}^N)$ is closed under convolution, under point-wise multiplication, translation and modulation, where the translation and modulation operators are defined by $\tau_y f(x):=f(x+y)$ and $M_t f(x):=e^{itx} f(x)$, respectively, where $t,x,y\in\mathbb{R}^N$  (\cite[Propositions 1.8.3 and 1.8.5]{B}).
	
	(3) The inclusions $\cD_\omega(\R^N)\hookrightarrow \cS_\omega(\R^N)\hookrightarrow \cE_\omega(\R^N)$ are continuous with dense range (\cite[Proposition 4.7.(1)]{BMT} and \cite[Propositions 1.8.6 and 1.8.7]{B}). Therefore, the inclusions $\mathcal{E}'_\omega(\mathbb{R}^N)\hookrightarrow\mathcal{S}'_\omega(\mathbb{R}^N)\hookrightarrow\mathcal{D}'_\omega(\mathbb{R}^N)$ are well defined and continuous.

(4)	The Fourier transform $\mathcal{F}:\mathcal{S}_{\omega}(\mathbb{R}^N)\to \mathcal{S}_{\omega}(\mathbb{R}^N)$ is a continuous isomorphism, that can be extended in the usual way to $\cS'_\omega(\R^N)$ (\cite[Proposition 1.8.2]{B}), i.e. $\mathcal{F}(T)(f):=\langle T,\hat{f}\rangle $ for every $f\in \cS_\omega(\R^N)$ and $T\in \cS'_\omega(\R^N)$. Moreover,  for every $f\in \cS_\omega(\R^N)$ and $T\in \cS'_\omega(\R^N)$ the convolution  $(T\star f)(x):=\langle T_y, \tau_x\check{f}\rangle$, for $x\in\R^N$, (where $\check{f}$ is the function $x\mapsto f(-x)$) is a well defined function on $\R^N$ such that  $T\star f \in \mathcal{S}'_\omega(\mathbb{R}^N)$, see \cite[Theorem 1.8.12]{B}, and 
	\begin{equation}\label{trf}
	\mathcal{F}(T\star f)=\hat{f}\mathcal{F}{T}.
	\end{equation}

	(5) The space $\cS_\omega(\R^N)$ is a nuclear Fr\'echet space, see, f.i., \cite[Theorem 3.3]{BJOS} or \cite[Theorem 1.1]{De1}.
\end{rem}

The space $\cS_\omega(\R^N)$ is a Fr\'echet space with different equivalent systems of seminorms. Indeed, the following result holds.

\begin{prop}\label{P.norme}
	Let $\omega$ be a non-quasianalytic weight function and consider $f \in \mathcal{S}(\mathbb{R}^N)$. Then $f\in\cS_\omega(\R^N)$ if and only if one of the following  conditions is satisfied.
	\begin{enumerate}	
		\item	\begin{enumerate}	
			\item[{\rm (i)} ] $\forall\lambda>0, \; \alpha\in\mathbb{N}^N_0, \; 1\leq p\leq \infty$ $\exists C_{\alpha,\lambda,p}>0$ such that 
			$\|\exp(\lambda\omega)\partial^\alpha f\|_p\leq C_{\alpha,\lambda,p}$, and
			
			\item[{\rm (ii)}] $\forall\lambda>0, \; \alpha\in\mathbb{N}^N_0,\; 1\leq p\leq \infty$  $\exists C_{\alpha,\lambda,p}>0$ such that 	
			$\| \exp(\lambda\omega)\partial^\alpha \hat{f}\|_p\leq C_{\alpha,\lambda,p} $.
		\end{enumerate}
		\item \begin{enumerate}
			\item[\rm (i)] $\forall \lambda>0,\; 1\leq p \leq \infty$ $\exists C_{\lambda,p}>0$ such that
			$\| \exp(\lambda\omega) f\|_p \leq C_{\lambda,p}$, and
			\item[\rm (ii)]	$\forall \lambda>0,\; 1\leq p \leq \infty$ $\exists C_{\lambda,p}>0$ such that
			$\|\exp(\lambda\omega)\hat{f}\|_p\leq C_{\lambda,p} $.
		\end{enumerate}
		\item $\forall\lambda,\;\mu>0,\; 1\leq p\leq \infty$ $\exists C_{\lambda,\mu,p}>0$ such that
		\[
		q_{\lambda,\mu,p}(f):=\underset{\alpha \in \mathbb{N}^N_0}{\sup}\| \exp(\mu\omega) \partial^\alpha f\|_p\exp\left(-\lambda\varphi^*_\omega\left(\frac{|\alpha|}{\lambda}\right)\right) \leq C_{\lambda,\mu,p} 	.\]
		\item $\forall\lambda,\;\mu>0,\; 1\leq p< \infty$ $\exists C_{\lambda,\mu,p}>0$ such that
		\[
		\sigma_{\lambda,\mu,p}(f):=\left(\sum_{\alpha\in\mathbb{N}^N_0}\|\exp(\mu\omega)\partial^\alpha f\|_p^p\exp\left(-p\lambda\varphi^*_\omega\left(\frac{|\alpha|}{\lambda}\right)\right)\right)^{\frac{1}{p}}\leq C_{\lambda,\mu,p}
			.\]
	\end{enumerate}
\end{prop}
\begin{proof} For a proof  of (1)$\Leftrightarrow$(2)$\Leftrightarrow$(3) we refer to \cite[Theorem 4.8]{BJOR} and \cite[Theorem 2.6]{BJO}.

	(3)$\Rightarrow$(4). Fix $\lambda, \mu>0$. Then by \eqref{eq.bb}
	\begin{align*}
	(\sigma_{\lambda,\mu,p}(f))^p&=\sum_{\alpha\in\mathbb{N}^N_0}\|\exp(\mu\omega)\partial^\alpha f\|_p^p\exp\left(-p\lambda\varphi^*_\omega\left(\frac{|\alpha|}{\lambda}\right)\right)\\
	&\leq\sum_{\alpha\in\mathbb{N}^N_0}\|\exp(\mu\omega)\partial^\alpha f\|_p^p\exp\left(-pL\lambda\varphi^*_\omega\left(\frac{|\alpha|}{L\lambda}\right)\right)\exp(-p|\alpha|+pL\lambda)\\
	&= \exp(pL\lambda)(q_{L\lambda,\mu,p}(f))^p\sum_{\alpha\in\mathbb{N}^N_0}\exp(-p|\alpha|),
	\end{align*}
	where $\sum_{\alpha\in\mathbb{N}^N_0}\exp(-p|\alpha|)<\infty$. So,
	\begin{equation*}
	\sigma_{\lambda,\mu,p}(f)\leq \exp(L\lambda)\left(\sum_{\alpha\in\mathbb{N}^N_0}\exp(-p|\alpha|)\right)^{\frac{1}{p}}q_{L\lambda,\mu,p}(f)<\infty.
	\end{equation*}
	
	(4)$\Rightarrow$(3). Fix $\lambda,\mu>0$. Then for every $\alpha\in\mathbb{N}^N_0$
	\begin{equation*}
	\|\exp(\mu\omega)\partial^\alpha f\|_p\exp\left(-\lambda\varphi^*_\omega\left(\frac{|\alpha|}{\lambda}\right)\right)\leq \sigma_{\lambda,\mu,p}(f).
	\end{equation*}
	Accordingly, $q_{\lambda,\mu,p}(f)\leq \sigma_{\lambda,\mu,p}(f)<\infty$. 
\end{proof}

\begin{rem}\label{R.Normenegativo} We observe that the inequalities
	\[
\sigma_{\lambda,\mu,p}(f)\leq 	\exp(L\lambda)\left(\sum_{\alpha\in\mathbb{N}^N_0}\exp(-p|\alpha|)\right)^{\frac{1}{p}}q_{L\lambda,\mu,p}(f) \ {\rm and \ }\ q_{\lambda,\mu,p}(f)\leq \sigma_{\lambda,\mu,p}(f)
	\]
	continue to hold in the case $\mu\in\R$ and their proofs are similar.
	\[
	\]
\end{rem}

Since $\cS_\omega(\R^N)$ is a Fr\'echet space, Proposition \ref{P.norme} implies that the sequences of norms
\begin{equation*}
 \{\sigma_{m,m,p}\}_{m\in\mathbb{N}} \ (1\leq p<\infty) \quad \text{and}\quad \{q_{m,m,p}\}_{m\in\mathbb{N}}\ (1\leq p\leq \infty),
\end{equation*}
 define  the same lc-topology on $\mathcal{S}_{\omega}(\mathbb{R}^N)$. In the following,  we will often use this  system of norms generating the Fr\'echet topology of $\cS_\omega(\R^N)$
\begin{equation}\label{eq.usualnorms}
q_{\lambda,\mu}(f):=q_{\lambda,\mu,\infty}(f), \ \lambda,\mu>0, \ f\in 	\cS_\omega(\R^N),
\end{equation}
or equivalently, the sequence of norms $\{q_{m,m}\}_{m\in\mathbb{N}}$.



\section{The spaces $\cO_{M,\omega}(\R^N)$ and $\cO_{C,\omega}(\R^N)$ and their duals}

The elements of the spaces $\cO_{M,\omega}(\R^N)$ and $\cO_{C,\omega}(\R^N)$ have been defined in \cite{AC} in terms of weighted $L^\infty$-norms. The main aim of 
this section is to show that the  elements of
such spaces can be also defined in  terms of weighted $L^p$-norms, with $1\leq p<\infty$. This characterization allows to easily show  some structure theorems for  
their strong dual spaces.  In order to do this, we begin by recalling the definition and some basic properties of the spaces 
$\cO_{M,\omega}(\R^N)$ and $\cO_{C,\omega}(\R^N)$ given in \cite{AC}. 

\begin{defn}\label{D.spaziO}
	Let $\omega$ be a non-quasianalytic weight function. 
	
	(a)  For $m\in \mathbb{N}$ and $n\in \mathbb{Z}$ we define the space $\mathcal{O}^m_{n,\omega}(\mathbb{R}^N)$ as the set of all functions $f\in C^\infty(\mathbb{R}^N)$ satisfying the following condition:
	\begin{equation}\label{omn}
	r_{m,n}(f):= \underset{\alpha\in \mathbb{N}^N_0}{\sup}\underset{x\in\mathbb{R}^N}{\sup}\, |\partial^\alpha f(x)|\exp\left(-n\omega(x)-m\varphi^*_\omega\left(\frac{|\alpha|}{m}\right)\right)<\infty.
	\end{equation}	
	The space $(\mathcal{O}^m_{n,\omega}(\mathbb{R}^N), r_{m,n})$ is a Banach space.
	
	(b)	We denote by $\mathcal{O}_{M,\omega}(\mathbb{R}^N)$  the set of all functions $f\in {C}^\infty(\mathbb{R}^N)$ such that for each $m\in \mathbb{N}$ there exist $C>0$ and $n \in\mathbb{N}$ such that for every $\alpha \in \mathbb{N}^N_0$ and $x\in \mathbb{R}^N$ we have
	\begin{equation}\label{m}
	|\partial^\alpha f(x)|\leq C\exp \left(n\omega(x)+m\varphi^*_\omega\left(\frac{|\alpha|}{m}\right)\right);		
	\end{equation}
	or equivalently,
	\begin{equation}\label{eq.OM}
	\cO_{M,\omega}(\R^N):=\bigcap_{m=1}^{\infty}\bigcup_{n=1}^\infty \mathcal{O}^m_{n,\omega}(\mathbb{R}^N).
	\end{equation}
	The elements of  $\mathcal{O}_{M,\omega}(\mathbb{R}^N)$ are called \textit{slowly increasing functions of Beurling type}. The space $\mathcal{O}_{M,\omega}(\mathbb{R}^N)$ is endowed with its natural lc-topology $t$, i.e., $\cO_{M,\omega}(\R^N)=\proj_{\stackrel{\leftarrow}{m}}\,\ind_{\stackrel{\rightarrow}{n}}\, \mathcal{O}^m_{n,\omega}(\mathbb{R}^N)$ is a projective limit of (LB)-spaces.
	
	(c)	We denote by  $\cO_{C,\omega}(\mathbb{R}^N)$  the set of all functions $f\in C^\infty(\R^N)$ for which there exists $n\in \mathbb{N}$ such that for every $m \in\mathbb{N}$ there exists $C>0$ so that for every $\alpha \in \mathbb{N}^N_0$ and $x\in \mathbb{R}^N$ we have
	\begin{equation}\label{c}
	|\partial^\alpha f(x)|\leq C\exp \left(n\omega(x)+m\varphi^*_\omega\left(\frac{|\alpha|}{m}\right)\right);		
	\end{equation}
	or equivalently,
	\begin{equation}\label{eq.OC}
	\cO_{C,\omega}(\R^N):=\bigcup_{n=1}^{\infty}\bigcap_{m=1}^\infty \mathcal{O}^m_{n,\omega}(\mathbb{R}^N).
	\end{equation}
	The elements of $\mathcal{O}_{C,\omega}(\mathbb{R}^N)$ are called  \textit{very slowly increasing functions of Beurling type}. The space $\cO_{C,\omega}(\R^N)$ is endowed with its natural lc-topology, i.e., $\cO_{C,\omega}(\R^N)=\ind_{\stackrel{\rightarrow}{n}}\,\proj_{\stackrel{\leftarrow}{m}}\, \mathcal{O}^m_{n,\omega}(\mathbb{R}^N)$ is an (LF)-space.
\end{defn}

We point out that condition $(\gamma)$ in Definition \ref{D.weight} implies that $\cO_{M,\omega}(\R^N)\su \cO_M(\R^N)$ ( $\cO_{C,\omega}(\R^N)\su \cO_C(\R^N)$, resp.) with continuous inclusion,  where $\cO_M(\R^N)$  ($\cO_C(\R^N)$, resp.) denotes the space of multipliers  (of convolutors, resp.)  of $\cS(\R^N)$.

 The space $\mathcal{O}_{M,\omega}(\mathbb{R}^N)$ is the space of multipliers of $\cS_\omega(\R^N)$ and of $\cS'_\omega(\R^N)$ as proved in \cite{AC}. In particular, in \cite[Theorem 5.3]{De} it is shown that $\mathcal{O}_{M,\omega}(\mathbb{R}^N)$ is an ultrabornological space.
 
On the other hand, by \cite[Theorems 4.4, 4.6 and 5.1]{AC} the following result holds.

 \begin{thm}\label{T.Multiplier} Let $\omega$ be a non-quasianalytic weight function and $f\in C^\infty(\R^N)$. Then the following properties are equivalent.
 	\begin{enumerate}
 		\item $f\in \cO_{M,\omega}(\R^N)$.
 		\item For every $g\in \cS_\omega(\R^N)$ and $m\in\N$ we have
 		\begin{equation}\label{eq.nuovenorme}
 		q_{m,g}(f):=\sup_{\alpha\in \N_0^N}\sup_{x\in \R^N}|g(x)||\partial^\alpha f(x)|<\infty.
 		\end{equation}
 		\item  For every $g\in \cS_\omega(\R^N)$ we have  $fg\in \cS_{\omega}(\R^N)$.
 		\item For every $T\in \mathcal{S}'_\omega(\mathbb{R}^N)$ we have $fT \in \mathcal{S}'_\omega(\mathbb{R}^N)$.  
 	\end{enumerate}
 	Moreover, if $f\in\cO_{M,\omega}(\R^N)$, then the linear operators $M_f: \mathcal{S}_{\omega}(\mathbb{R}^N)\to\mathcal{S}_{\omega}(\mathbb{R}^N)$ defined by $M_f(g):=fg$, for  $g\in\mathcal{S}_{\omega}(\mathbb{R}^N)$, and $\cM_f\colon \mathcal{S}'_\omega(\mathbb{R}^N)\to  \mathcal{S}'_\omega(\mathbb{R}^N)$ defined by $\cM_f(T):=fT$, for $T\in \mathcal{S}'_\omega(\mathbb{R}^N)$, are continuous.
 \end{thm}

The set $\{q_{m,g}\}_{m\in\N,g\in \cS_\omega(\R^N)}$ defines a complete Hausdorff lc-topology $\tau$ on $\mathcal{O}_{M,\omega}(\mathbb{R}^N)$ weaker than $t$ (\cite[Theorem 5.2(2)]{AC}). Actually, by combining \cite[Proposition 5.6 and Theorem 5.9]{AC} with \cite[Theorem 5.2]{De} it follows that $t=\tau$. So, the lc-topology  $t$ is described by means of $\{q_{m,g}\}_{m\in\N,g\in \cS_\omega(\R^N)}$. Moreover, by \cite[Theorems 3.8, 3.9 and 5.2(1)]{AC}  the following relationship between the spaces 	$\mathcal{O}_{C,\omega}(\mathbb{R}^N)$ and $\mathcal{O}_{M,\omega}(\mathbb{R}^N)$ holds.
 
\begin{thm}\label{T.incl} 	Let $\omega$ be a non-quasianalytic weight function. 
		Then the  inclusions
	    \begin{equation}\label{eq.incl_comp}
    	\mathcal{D}_{\omega}(\mathbb{R}^N) \hookrightarrow \mathcal{S}_{\omega}(\mathbb{R}^N) \hookrightarrow \mathcal{O}_ {C,\omega}(\mathbb{R}^N)\hookrightarrow \mathcal{O}_ {M,\omega}(\mathbb{R}^N)\hookrightarrow \mathcal{E}_ {\omega}(\mathbb{R}^N)
    	\end{equation}
	    are well-defined, continuous and with dense range.
\end{thm}

Denoted by $\cO'_{C,\omega}(\R^N)$ ($\cO'_{M,\omega}(\R^N)$, resp.) the strong dual space of  $\cO_{C,\omega}(\R^N)$ ($\cO_{M,\omega}(\R^N)$, resp.),  Theorem \ref{T.incl} implies that the inclusions
	\begin{equation}\label{eq.inc.duali}
		\mathcal{E}'_{\omega}(\mathbb{R}^N) \rightarrow \mathcal{O}'_ {M,\omega}(\mathbb{R}^N)\hookrightarrow \mathcal{O}'_ {C,\omega}(\mathbb{R}^N) \rightarrow  \mathcal{S}'_{\omega}(\mathbb{R}^N)\rightarrow \mathcal{D}'_ {\omega}(\mathbb{R}^N)  
	\end{equation}
are also well-defined and continuous. Moreover, the following result holds.

\begin{prop}\label{P.duali-Inc} Let $\omega $ be a non-quasianalytic weight function. Then the following properties hold.
	\begin{enumerate}
		\item $\cO_{M,\omega}(\R^N)\hookrightarrow \cS'_\omega(\R^N)$ continuously.
		\item $\cO_{C,\omega}(\R^N)\hookrightarrow \cS'_\omega(\R^N)$ continuously.
	\end{enumerate}
\end{prop}

\begin{proof} In view of Theorem \ref{T.incl} it suffices to prove only the first statement. To see this, we first show  for every  $m, n\in\N$ that the inclusion $\cO_{n,\omega}^m(\R^N)\hookrightarrow \cS'_\omega(\R^N)$ is continuous. Indeed, fixed $m,n\in\N$ and $f\in \cO_{n,\omega}^m(\R^N)$ and denoted by $T_f$ the linear functional defined as $T_f(g):=\int_{\R^N}f(x)g(x)\,dx$, for $g\in\ \cS_\omega(\R^N)$, we have for every $g\in\cS_\omega(\R^N)$ that 
	\begin{equation}\label{eq.Lin}
	|T_f(g)|\leq \int_{\R^N}|f(x)g(x)|\,dx\leq r_{m,n}(f)\sigma_{m,n,1}(g).
	\end{equation}
This means that $T_f\in \cS'_\omega(\R^N)$, i.e., $f\in \cS'_\omega(\R^N)$. Fixed a bounded subset of $\cS_{\omega}(\R^N)$, from \eqref{eq.Lin} it  follows  that 
\[
\sup_{g\in B}	|T_f(g)|\leq \sup_{g\in B}\sigma_{m,n,1}(g) r_{m,n}(f),
\]
where $\sup_{g\in B}\sigma_{m,n,1}(g)<\infty $. Hence, as $f$ is arbitrary, we can conclude  that the inclusion $\cO_{n,\omega}^m(\R^N)\hookrightarrow \cS'_\omega(\R^N)$ is continuous.  

Now,  the fact that the inclusion $\cO_{n,\omega}^m(\R^N)\hookrightarrow \cS'_\omega(\R^N)$ is continuous for every $m,n\in\N$ yields that the inclusion 
$
\bigcup_{n=1}^\infty\cO_{n,\omega}^m(\R^N)
\hookrightarrow \cS'_\omega(\R^N)
$
is also continuous for every $m\in\N$ and hence, the inclusion $\cO_{M,\omega}(\R^N)\hookrightarrow \cS'_\omega(\R^N)$ is continuous.
\end{proof}

In order to characterize the elements of the spaces	  $\mathcal{O}_{M,\omega}(\mathbb{R}^N)$ and $\mathcal{O}_{C,\omega}(\mathbb{R}^N)$ in terms of weighted $L^p$-norms, we first observe the following fact.

\begin{prop}\label{newdefomoc}
	Let $\omega$ be a non-quasianalytic weight function. Then the following properties hold.
	\begin{enumerate}
			\item $f\in \mathcal{O}_{M,\omega}(\mathbb{R}^N)$ if and only if $f\in C^\infty(\mathbb{R}^N)$ and for each $m\in\mathbb{N}$ there exist  $C>0$ and  $n\in \mathbb{N}$ such that for every $\alpha\in\mathbb{N}^N_0$ and $x\in\mathbb{R}^N$ we have  
		\begin{equation}\label{eq.nuoM}
		|\partial^\alpha f(x)|\leq C\exp \left(n\omega(x)+m\varphi^*_\omega\left(\frac{|\alpha|}{m}\right)-|\alpha|\right).	
		\end{equation}
		\item $f\in \mathcal{O}_{C,\omega}(\mathbb{R}^N)$ if and only if $f\in C^\infty(\mathbb{R}^N)$ and there exists $n\in \mathbb{N}$ such that for every $m\in\mathbb{N}$ there exists $C>0$ so that for every $\alpha\in\mathbb{N}^N_0$ and $x\in\mathbb{R}^N$ we have
		\begin{equation}\label{ndoc}
		|\partial^\alpha f(x)|\leq C\exp \left(n\omega(x)+m\varphi^*_\omega\left(\frac{|\alpha|}{m}\right)-|\alpha|\right).	
		\end{equation}
	\end{enumerate}
\end{prop}
\begin{proof}
	We prove only the assertion (2). The proof of the assertion (1) is similar with obviuos changes regarding quantifiers.

Suppose that the necessary condition is satisfied.	Since $m\varphi^*_\omega\left(\frac{|\alpha|}{m}\right)-|\alpha|\leq m\varphi^*_\omega\left(\frac{|\alpha|}{m}\right)$ for every $\alpha\in\N_0^N$ and $m\in\N$, the inequality in $(\ref{c})$ holds.

 Conversely, suppose that $f\in \cO_{C,\omega}(\R^N)$. Then there is $n\in\mathbb{N}$ such that the inequality in  $(\ref{c})$ is satisfied for every $m\in\mathbb{N}$ with $C>0$ depending on $m$. So, fixed $m\in\mathbb{N}$ and choosen $m'\in\mathbb{N}$ so that $m'\geq Lm$, where $L\geq 1$ is the constant appearing in \eqref{l}, there is $C>0$ so that for every $\alpha\in\mathbb{N}^N_0$ and $x\in\mathbb{R}^N$ we have
	\begin{equation*}
	|\partial^\alpha f(x)|\leq C\exp \left(n\omega(x)+m'\varphi^*_\omega\left(\frac{|\alpha|}{m'}\right)\right).			\end{equation*}
	Therefore, taking into account that $\frac{\varphi^*_\omega(t)}{t}$ is an increasing function in $(0,\infty)$ and applying inequality ($\ref{eq.bb}$), it follows  for every $\alpha\in\mathbb{N}^N_0$ and $x\in\mathbb{R}^N$ that
	\begin{align*}
	|\partial^\alpha f(x)|&\leq C\exp \left(n\omega(x)+m'\varphi^*_\omega\left(\frac{|\alpha|}{m'}\right)\right)\leq  C\exp \left(n\omega(x)+	Lm\varphi^*_\omega\left(\frac{|\alpha|}{Lm}\right)\right)\\&\leq C\exp(Lm)\exp \left(n\omega(x)+	m\varphi^*_\omega\left(\frac{|\alpha|}{m}\right)-|\alpha|\right).
	\end{align*}
	This completes the proof.
\end{proof}

\begin{rem}\label{normap}
Suppose that the function $f\in C^\infty(\mathbb{R}^N)$ satisfies the condition
	\begin{equation*}
	|\partial^\alpha f(x)|\leq C\exp \left(n\omega(x)+m\varphi^*_\omega\left(\frac{|\alpha|}{m}\right)-|\alpha|\right),		\end{equation*}
	for every $\alpha\in\mathbb{N}^N_0$, $x\in\mathbb{R}^N$ and for some $n,m\in\mathbb{N}$.  Then, for a fixed $1\leq p<\infty$, the function $f$ clearly satisfies the condition
	\begin{equation*}
	\exp(-(n+n_0)\omega(x))|\partial^\alpha f(x)|\leq C\exp \left(-n_0\omega(x)+m\varphi^*_\omega\left(\frac{|\alpha|}{m}\right)-|\alpha|\right),	\end{equation*}
	for every $\alpha\in\mathbb{N}^N_0$  and $x\in\mathbb{R}^N$ and for $n_0:= \left[\frac{N+1}{bp}\right]+1>\frac{N+1}{bp}$, where $b$ is the constant appearing in condition ($\gamma$). Since $\exp(-n_0\omega)\in L^p(\R^N)$ by \eqref{eq.Lpspazi}, it follows for every $\alpha\in\mathbb{N}^N_0$ that 
	\begin{equation*}
	\|\exp(-(n+n_0)\omega)\partial^\alpha f\|_p\leq C\|\exp \left(-n_0\omega\right)\|_p\exp\left(m\varphi^*_\omega\left(\frac{|\alpha|}{m}\right)-|\alpha|\right).		
	\end{equation*}
		Accordingly, we obtain that
	\begin{equation*}
	\sum_{\alpha\in\mathbb{N}^N_0} \|\exp(-(n+n_0)\omega)\partial^\alpha f\|_p^p\exp\left(-pm\varphi^*_\omega\left(\frac{|\alpha|}{m}\right)\right)\leq (CD)^p\sum_{\alpha\in\mathbb{N}^N_0}\exp(-p|\alpha|)<\infty,
	\end{equation*}
	where $D:=\|\exp\left(-n_0\omega\right)\|_p$.	
\end{rem}

In view of Remark \ref{normap} above it is natural to introduce the  following spaces of $C^\infty$ functions on $\R^N$.
 
\begin{defn}\label{D.spaziOp}
	Let $\omega$ be a non-quasianalytic weight function and $1\leq p<\infty$.
	 For   $m\in \mathbb{N}$ and $n\in \mathbb{Z}$ we define the space $\mathcal{O}^m_{n,\omega,p}(\mathbb{R}^N)$ as the set of all functions $f\in C^\infty(\mathbb{R}^N)$ satisfying the following condition:
	\begin{equation}\label{omnp}
	r^p_{m,n,p}(f):= \sum_{\alpha\in\mathbb{N}^N_0} \|\exp(-n\omega)\partial^\alpha f\|_p^p\exp\left(-mp\varphi^*_\omega\left(\frac{|\alpha|}{m}\right)\right)<\infty.
	\end{equation}
	The space $(\cO_{n,\omega,p}^m(\R^N),r_{m,n,p})$ is a Banach space.
	\end{defn}

We observe that 
 for every $n, n'\in \mathbb{Z}$ with $n\leq n'$ and $m\in \mathbb{N}$, the inclusion
	$
	\mathcal{O}^m_{n,\omega,p}(\mathbb{R}^N)\hookrightarrow\mathcal{O}^m_{n',\omega,p}(\mathbb{R}^N)$
		is continuous and that 
	 for every $n\in \mathbb{Z}$ and $ m, m'\in \mathbb{N}$ with $m\leq m'$, the inclusion
	$
	\mathcal{O}^{m'}_{n,\omega,p}(\mathbb{R}^N)\hookrightarrow\mathcal{O}^m_{n,\omega,p}(\mathbb{R}^N)$
	is also continuous. So, we can set

	\begin{defn}\label{D.SapziOOP} Let $\omega$ be a non-quasianalytic weight function and $1\leq p<\infty$. 
		
	(a) We define the space $\cO_{M,\omega,p}(\R^N)$ by
	\begin{equation}\label{eq.Omp}
	\cO_{M,\omega,p}(\R^N):=\bigcap_{m=1}^{\infty}\bigcup_{n=1}^\infty \mathcal{O}^m_{n,\omega,p}(\mathbb{R}^N).
	\end{equation}
	The space $\cO_{M,\omega,p}(\R^N)$ is endowed with its natural lc-topology, i.e., $\cO_{M,\omega,p}(\R^N) =\proj_{\stackrel{\leftarrow}{m}}\,\ind_{\stackrel{\rightarrow}{n}}\,\mathcal{O}^m_{n,\omega,p}(\mathbb{R}^N)$ is a projective limit of (LB)-spaces.

	(b) We define the space  $\cO_{C,\omega,p}(\R^N)$  by
	\begin{equation}\label{eq.Ocp}
	\cO_{C,\omega,p}(\R^N):=\bigcup_{n=1}^{\infty}\bigcap_{m=1}^\infty \mathcal{O}^m_{n,\omega,p}(\mathbb{R}^N).
	\end{equation}
	The space  $\cO_{C,\omega,p}(\R^N)$ is endowed with its natural lc-topology, i.e.,  $\cO_{C,\omega,p}(\R^N)=\ind_{\stackrel{\rightarrow}{n}}\proj_{\stackrel{\leftarrow}{m}}\cO_{n,\omega,p}^m(\R^N)$ is an (LF)-space.
\end{defn}

For every $p\in [1,\infty)$,  $\cO_{M,\omega}(\R^N)=\cO_{M,\omega,p}(\R^N)$ and $\cO_{C,\omega}(\R^N)=\cO_{C,\omega,p}(\R^N)$ algebraically and topologically.  In order to show this, we first establish  the following result.

\begin{prop}\label{P.InclLP} Let $\omega$ be a non-quasianalytic weight function and $1\leq p<\infty$. Then the following properties are satisfied. 
	\begin{enumerate}
\item $\mathcal{O}_{M,\omega}(\mathbb{R}^N)\hookrightarrow\mathcal{O}_{M,\omega,p}(\mathbb{R}^N)$ {continuously}.
\item  $\mathcal{O}_{C,\omega}(\mathbb{R}^N)\hookrightarrow \cO_{C,\omega,p}(\R^N)$  {continuously}.
\end{enumerate}
\end{prop}

\begin{proof} (1) Fix $m,  n\in\N$ and set $n_0:=\left[\frac{N+1}{pb}\right]+1$. Then Proposition \ref{newdefomoc}(1) and Remark \ref{normap} imply  for every $f\in \cO^m_{n,\omega}(\R^N)$ and $m'\geq Lm$ that
	\begin{equation}\label{eq.CSupLp}
	r_{m', n+n_0,p}(f)\leq \exp(Lm)Cr_{m,n}(f),
	\end{equation}
where $C:=||\exp(-n_0\omega)||_p\left(\sum_{\alpha\in\N_0^N}\exp(-p|\alpha|)\right)^{\frac{1}{p}}<\infty$. Accordingly, the inclusions 
\[
\cO^m_{n,\omega}(\R^N)\hookrightarrow \cO^{m'}_{n+n_0,\omega,p}(\R^N)\hookrightarrow \bigcup_{n'=1}^\infty \cO^{m'}_{n',\omega,p}(\R^N)
\]
 are continuous for every $m'\geq Lm$. The arbitrarity of $n\in\N$ yields that also the inclusion
\[
\bigcup_{n=1}^\infty\cO^m_{n,\omega}(\R^N)\hookrightarrow  \bigcup_{n'=1}^\infty \cO^{m'}_{n',\omega,p}(\R^N)
\] 
is continuous for every  $m'\geq Lm$. Finally, since $m\in\N$ is  arbitrary and the spaces $\cO_{M,\omega}(\R^N)$ and $\mathcal{O}_{M,\omega,p}(\mathbb{R}^N)$ are endowed with the projective lc-topology defined by the spectrum $\{\bigcup_{n=1}^\infty\cO^m_{n,\omega}(\R^N)\}_{m\in\N}$ and $\{ \bigcup_{n'=1}^\infty \cO^{m'}_{n',\omega,p}(\R^N)\}_{m'\in\N}$ respectively,  the thesis  follows. 
 
 (2) Fix $m,  n\in\N$ and set $n_0:=\left[\frac{N+1}{pb}\right]+1$. Then Proposition \ref{newdefomoc}(2) and Remark \ref{normap} clearly imply  for every $f\in \cO^m_{n,\omega}(\R^N)$, $m'\geq Lm$ and $n'\geq n+n_0$ that
 \begin{equation}\label{eq.CSupLp-2}
 r_{m', n',p}(f)\leq \exp(Lm)Cr_{m,n}(f).
 \end{equation}
This means that the inclusion
 \[
 \cO^m_{n,\omega}(\R^N)\hookrightarrow \cO^{m'}_{n',\omega,p}(\R^N)
 \]
 is continuous for every $m'\geq Lm$ and $n'\geq n+n_0$.
Since $m\in \N$ is arbitrary and  the spaces $\bigcap_{m=1}^\infty \cO^m_{n,\omega}(\R^N)$ and $\bigcap_{m'=1}^\infty \cO^{m'}_{n',\omega}(\R^N)$ are endowed with 
 the projective lc-topology defined by the spectrum $\{\cO^m_{n,\omega}(\R^N)\}_{m\in\N}$ and $\{\cO^{m'}_{n',\omega,p}(\R^N)\}_{m\in\N}$ respectively, it follows that also the inclusion
 \[
 \bigcap_{m=1}^\infty \cO^m_{n,\omega}(\R^N)\hookrightarrow \bigcap_{m'=1}^\infty \cO^{m'}_{n',\omega,p}(\R^N)
 \]
is continuous for every $n'\geq n+n_0$. Finally, taking into account that $\bigcap_{m'=1}^\infty \cO^{m'}_{n',\omega,p}(\R^N)\hookrightarrow \cO_{C,\omega,p}(\R^N)$  and that $n\in\N$ is arbitrary, the thesis follows.
\end{proof}

\begin{rem}\label{Omnp.rappS} 
	We observe that if $f\in \cS_\omega(\R^N)$, then $\sigma_{m,n,p}(f)=r_{m,-n,p}(f)$ for every $m\in \mathbb{N}$,  $n\in\mathbb{N}$ and $p\in [1,\infty)$. Therefore, by Proposition \ref{P.norme} it follows that
$\mathcal{S}_{\omega}(\mathbb{R}^N)=\bigcap_{n=1}^{\infty}\bigcap_{m=1}^\infty \mathcal{O}^m_{-n,\omega,p}(\mathbb{R}^N)$ for every $p\in [1,\infty)$.	
\end{rem}

We now show the reverse inclusions, i.e., that 
$\cO_{M,\omega,p}(\R^N)\hookrightarrow  \cO_{M,\omega}(\R^N)$ and $\cO_{C,\omega,p}(\R^N)\hookrightarrow \cO_{C,\omega}(\R^N)$ continuously for every $p\in [1,\infty)$.

In order to prove such topological inclusions,   we introduce the weighted space  $W^{k,p}(\mathbb{R}^N, \exp(-n\omega(x))\,dx)$, with $1\leq p\leq \infty$, $n\in\N_0$ and $k\in\N_0 \cup \{\infty\}$, defined as the set of all functions $f\in W^{k,p}_{loc}(\mathbb{R}^N)$ such that
\begin{equation*}
\|f\|_{k,p,\exp(-n\omega)}:=\sum_{|\alpha|\leq k} \|\exp(-n\omega)\partial^\alpha f\|_p<\infty.
\end{equation*}
Since the weight $\exp(-n\omega(x))\in L^\infty(\R^N)$ is a positive function on $\R^N$, it is straithforward to verify that $( W^{k,p}(\mathbb{R}^N, \exp(-n\omega(x))\,dx), \|\cdot\|_{k,p,\exp(-n\omega)})$ is a Banach space
and that  $C^k_0(\R^N)$ is a dense subspace of $( W^{k,p}(\mathbb{R}^N, \exp(-n\omega(x))\,dx), \|\cdot\|_{k,p,\exp(-n\omega)})$.


We now show that  an appropriate embedding theorem is also valid in the setting of the spaces introduced above.

\begin{prop}\label{PropMorrey}
	Let $\omega$ be a non-quasianalytic weight function and let $k\in\N$ and $1\leq p <\infty$. If $kp>N$, then  for each  $n\in\N_0$ there exist  $n'\geq n$ with $n'=n'(n,\omega)\in\N_0$  and   $C>0$ with $C=C(n,N,k,p)$ such that  for every $f\in W^{k,p}(\mathbb{R}^N, \exp(-n\omega(x))\,dx) $ the following inequality is valid:
	\begin{equation}\label{morrey}
	\|f\exp(-n'\omega)\|_\infty\leq C\|f\|_{k,p,\exp(-n\omega)}.
	\end{equation} 
\end{prop}

\begin{proof} We first consider the case $k=1$ and so $p>N$. 
	
By  Morrey's inequality (\cite[Theorem 4, p.266]{E}), there exists a constant $C_N>0$ such that for every $f\in C^1(\R^N)$, $x\in \R^N$ and $r>0$ we have
	\begin{equation}\label{eq.Morrey-1}
	\frac{1}{|B_r(x)|}\int_{B_r(x)} |f(y)-f(x)|\,dy\leq C_N \int_{B_r(x)} \frac{|Df(y)|}{|y-x|^{N-1}}\, dy.
	\end{equation}
Now, let $n'\geq Kn$ with $n'\in\N_0$, where $K\geq 1$ is the constant appearing in Definition \ref{D.weight}($\alpha$). Since  $\omega(y)=\omega(|(y-x)+x)|)\leq \omega(|x-y|+|x|))\leq K(\omega(x-y)+\omega(x)+1)\leq K(\omega(1)+\omega(x)+1)$ for $x,y\in\R^N$ with $|x-y|\leq 1$, it follows by \eqref{eq.Morrey-1} for every   $f\in C^1(\R^N)$ and  $x\in\R^N$, that 
	\begin{align*}
&	|f(x)|\exp(-n'\omega(x))\leq |f(x)|\exp(-Kn\omega(x))= \frac{1}{|B_1(x)|}\int_{B_1(x)} |f(x)|\exp(-Kn\omega(x))\,dy\\
&\quad \leq  \frac{1}{|B_1(x)|}\int_{B_1(x)} |f(x)-f(y)+f(y)|\exp(-Kn\omega(x)\, dy\\
&\quad  \leq  C_N\int_{B_1(x)} \frac{|Df(y)|\exp(-Kn\omega(x))}{|y-x|^{N-1}}\,dy
+\frac{1}{|B_1(x)|}\int_{B_1(x)}|f(y)|\exp(-Kn\omega(x))dy\\
&\quad \leq  C_N\exp(Kn(1+\omega(1)))\int_{B_1(x)} \frac{|Df(y)|\exp(-n\omega(y))}{|y-x|^{N-1}}\,dy\\
&\quad +\frac{\exp(Kn(1+\omega(1)))}{|B_1(x)|}\int_{B_1(x)}|f(y)|\exp(-n\omega(y))dy.
	\end{align*}
So, setting $C:=\max\left\{C_N\exp(Kn(1+\omega(1))), \frac{\exp(Kn(1+\omega(1)))}{|B_1(x)|}\right\}$ and  applying H\"older's inequality, we get for every  $f\in C^1_0(\R^N)$ and $x\in\R^N$ that 
	\begin{align*}
&	|f(x)|\exp(-n'\omega(x))\leq \\
&\quad\leq C\left(\int_{B_1(x)} |Df(y)|^p\exp(-pn\omega(y))\,dy\right)^\frac{1}{p}\left(\int_{B_1(x)} \frac{1}{|y-x|^{(N-1)p'}}\,dy\right)^\frac{1}{p'}\\
&\quad+C|B_1(x)|^{1/p'}\left(\int_{B_1(x)}|f(y)|^p\exp(-pn\omega(y))dy\right)^{\frac{1}{p}}\leq C' \|f\|_{1,p,\exp(-n\omega)},
	\end{align*}
	after having observed  that $\left(\int_{B_1(x)} \frac{1}{|y-x|^{(N-1)p'}}\,dy\right)^\frac{1}{p'}<\infty$ as $p>N$. Therefore, we have  for every $f\in C^1_0(\R^N)$ that 
	 \[
	\|f\exp(-n'\omega)\|_\infty\leq C' \|f\|_{1,p,\exp(-n\omega)}.
	\]	
Thus, \eqref{morrey} is proved for $k=1$ as $C_0^1(\R^N)$ is a dense subspace of $W^{1,p}(\R^N,\exp(-n\omega(x))dx)$. To conclude the proof in the case $k>1$, we  proceed as follows.

If $k>1$ but $p>N$,   we have by the result proved above that for every $f\in W^{k,p}(\R^N,\exp(-n\omega(x))dx)$ and $x\in\R^N$  
\begin{align*}
&	|f(x)|\exp(-n'\omega(x))\leq \\
&\quad\leq C'\left[\left(\int_{B_1(x)}|f(y)|^p\exp(-pn\omega(y))dy\right)^{\frac{1}{p}}+\left(\int_{B_1(x)} |Df(y)|^p\exp(-pn\omega(y))\,dy\right)^\frac{1}{p}\right]\\
&\quad\leq C''\sum_{|\alpha|\leq k}\left(\int_{B_1(x)}|\partial^\alpha f(y)|^p\exp(-pn\omega(y))dy\right)^{\frac{1}{p}}\leq C'' ||f||_{k,p,\exp(-n\omega)},
\end{align*}
with $C''=C''(C',N)>0$. Accordingly, we obtain for every $f\in W^{k,p}(\R^N,\exp(-n\omega(x))dx)$ that 
\[
||f\exp(-n'\omega)||_\infty\leq C''||f||_{k,p,\exp(-n\omega)}.
\]
If $p\leq N<kp$, then there exists $j\in\N$ with $1\leq j\leq k-1$ such that $jp\leq N<(j+1)p$. If $jp<N$, we set $r:=\frac{Np}{N-jp}$. If $jp=N$, we choose $r>\max\{N,p\}$. In both cases, $r>N$ and $r\geq p$. So, by the result proved above we have for every $f\in W^{k,p}(\R^N,\exp(-n\omega(x))dx)$ and $x\in\R^N$ that 
\begin{align*}
&	|f(x)|\exp(-n'\omega(x))\leq \\
&\quad\leq C'_1\left[\left(\int_{B_1(x)}|f(y)|^r\exp(-rn\omega(y))dy\right)^\frac{1}{r}+\left(\int_{B_1(x)} |Df(y)|^p\exp(-rn\omega(y))\,dy\right)^\frac{1}{r}\right]\\
&\quad \leq C_1''\sum_{|\alpha|\leq k}\left(\int_{B_1(x)}|\partial^\alpha f(y)|^r\exp(-rn\omega(y))dy\right)^\frac{1}{r}\\
&\quad \leq C_1''c\sum_{|\alpha|\leq k}\left(\int_{B_1(x)}|\partial^\alpha f(y)|^p\exp(-pn\omega(y))dy\right)^\frac{1}{p} \leq C''_1c ||f||_{k,p,\exp(-n\omega)},
\end{align*}
after having observed that the map  $L^r(B_1(x))\hookrightarrow L^p(B_1(x))$ is continuous as $r\geq p$ with norm $c$ depending only on $r,p$ and on the volume $|B_1(x)|=|B_1(0)|$. So, also in this case we have for every $f\in W^{k,p}(\R^N,\exp(-n\omega(x))dx)$ that 
\[
||f\exp(-n'\omega)||_\infty\leq C''_1c||f||_{k,p,\exp(-n\omega)},
\]
with $C''_1c>0$ depending on $n, N, k,p$. So, the proof is complete.
\end{proof}

\begin{rem}
	If the weight function $\omega$ is sub-additive, i.e., $\omega(s+t)\leq \omega(s)+\omega(t)$ for $s,t\geq 0$, from the proof above it follows that $||f\exp(-n\omega)||_\infty\leq C||f||_{k,p,\exp(-n\omega)}$ whenever $n\in\N_0$ and $f\in W^{k,p}(\mathbb{R}^N, \exp(-n\omega(x))\,dx)$ with $kp>N$.
\end{rem}

Proposition \ref{PropMorrey} is the main tool towards the following result.

\begin{prop}\label{PropInclp}
	Let $\omega$ be a non-quasianalytic weight function and let $ 1\leq p<\infty$. Then for every  $n\in\mathbb{N}$ there exists $n'\geq n$ such that for every $m\in\N$ the inclusion	
	\begin{equation}\label{eq.inlcP}
	\mathcal{O}^{2m}_{n,\omega,p}(\mathbb{R}^N)\hookrightarrow\mathcal{O}^m_{n',\omega}(\mathbb{R}^N)
	\end{equation}
	is well-defined and continuous.
\end{prop}

\begin{proof} Fix $n\in\N$ and choose $k\in\N$ satisfying   $kp>N$. Then by Proposition \ref{PropMorrey} there exist $n'=n'(n,\omega)\geq n$ with $n\in\N$ and $C=C(n,N,k,p)>0$ such that inequality \eqref{morrey} is satisfied. So, for a fixed $m\in\N$, we have for every $f\in \mathcal{O}^{2m}_{n,\omega,p}(\mathbb{R}^N)\subset C^\infty(\R^N)$ and  $\alpha\in\mathbb{N}^N$ that 
	\begin{equation*}
	\|\partial^\alpha f \exp(-n'\omega)\|_\infty\leq C\|\partial^\alpha f\|_{k,p,\exp(-n\omega)}.
	\end{equation*}
Thus, applying the fact that  $\varphi^*_\omega(t)/t$ is increasing function in $(0,\infty)$ and inequality \eqref{secondprop}, it follows for every $f\in \mathcal{O}^{2m}_{n,\omega,p}(\mathbb{R}^N)$ and  $\alpha\in\mathbb{N}^N$ that 
	\begin{align*}
	&\|\partial^\alpha f \exp(-n'\omega)\|_\infty\exp\left(-m\varphi^*_\omega\left(\frac{|\alpha|}{m}\right)\right)\leq C\|\partial^\alpha f\|_{k,p\exp(-n\omega)}\exp\left(-m\varphi^*_\omega\left(\frac{|\alpha|}{m}\right)\right)\\
	&\quad =C\|\partial^\alpha f\exp(-n\omega)\|_p\exp\left(-m\varphi^*_\omega\left(\frac{|\alpha|}{m}\right)\right)\\
	&\quad +C\sum_{0<|\beta|<k} \|\partial^{\alpha+\beta}f\,\exp(-n\omega)\|_p\exp\left(-m\varphi^*_\omega\left(\frac{|\alpha|}{m}\right)\right)\\
	&\quad\leq C\|\partial^\alpha f\exp(-n\omega)\|_p\exp\left(-2m\varphi^*_\omega\left(\frac{|\alpha|}{2m}\right)\right)\\
	&\quad +C\sum_{0<|\beta|<k} \|\partial^{\alpha+\beta}f\,\exp(-n\omega)\|_p\exp\left(-2m\varphi^*_\omega\left(\frac{|\alpha+\beta|}{2m}\right)+m\varphi^*_\omega\left(\frac{\beta}{m}\right)\right)\\
	&\quad \leq C'\left(\sum_{\gamma\in\mathbb{N}^N}\|\partial^\gamma f\exp(-n\omega)\|_p^p\exp\left(-2pm\varphi^*_\omega\left(\frac{|\gamma|}{2m}\right)\right)\right)^{\frac{1}{p}}=C'r_{2m,n,p}(f),	
	\end{align*}
	where $C':=C\left(\sum_{0<|\beta|<k}\exp\left(p'm\varphi^*_\omega\left(\frac{|\beta|}{m}\right)\right)+1\right)^{\frac{1}{p'}}$ whenever $1<p<\infty$, with $\frac{1}{p}+\frac{1}{p'}=1$, and $C':=C\sup_{0<|\beta|<k}\exp\left(p'm\varphi^*_\omega\left(\frac{|\beta|}{m}\right)\right)+1$ whenever $p=1$. Accordingly, for every $f\in \cO_{n,\omega,p}^{2m}(\R^N)$ the following inequality holds
	\begin{equation*}
	r_{m,n'}(f)=\underset{\alpha\in\mathbb{N}^N_0}{\sup}\,\|\partial^\alpha f \exp(-n\omega)\|_\infty\exp\left(-m\varphi^*_\omega\left(\frac{|\alpha|}{m}\right)\right)\leq C'r_{2m,n,p}(f).
	\end{equation*}
	This completes the proof.
\end{proof}

Thanks to Proposition  \ref{PropInclp}, we are now able to show the following result.

\begin{prop}\label{incoc}
	Let $\omega$ be a non-quasianalytic weight function and  $ 1\leq p< \infty$. Then the following properties are satisfied.
	\begin{enumerate}
		\item $\cO_{M,\omega}(\R^N)=\cO_{M,\omega,p}(\R^N)$ algebraically and topologically.
		\item $\cO_{C,\omega}(\R^N)=\cO_{C,\omega,p}(\R^N)$ algebraically and topologically.
	\end{enumerate}
\end{prop}

\begin{proof}  By Proposition \ref{P.InclLP} it suffices to establish that both inclusions
\[
\cO_{M,\omega,p}(\R^N)\hookrightarrow \cO_{M,\omega}(\R^N)\ \  {\rm and} \ \ \cO_{C,\omega,p}(\R^N)\hookrightarrow \cO_{C,\omega}(\R^N)
\]	
	are well-defined and continuous. But this easily follows by applying  Proposition \ref{PropInclp} and by argumenting in a similar way as in the proof of Proposition \ref{P.InclLP}. Indeed, Proposition \ref{PropInclp} implies that  the inclusions
	\begin{equation}\label{eq.inLpp}
	\bigcup_{n=1}^\infty \cO^{2m}_{n,\omega,p}(\R^N)\hookrightarrow \bigcup_{n=1}^\infty \cO^{m}_{n,\omega}(\R^N), \quad m\in\N,
	\end{equation}
	\begin{equation}\label{eq.inLLp}
		\bigcap_{m=1}^\infty \cO^{m}_{n,\omega,p}(\R^N)\hookrightarrow \bigcap_{m=1}^\infty \cO^{m}_{n',\omega}(\R^N), \quad n'\geq n,
	\end{equation}
	are continuous and hence the result follows.
%
%
\end{proof}

 Denoting by $\cO'_{M,\omega,p}(\R^N)$ ($\cO'_{C,\omega,p}(\R^N)$, resp.), for $1\leq p<\infty$, the strong dual of $\cO_{M,\omega,p}(\R^N)$ ($\cO_{C,\omega,p}(\R^N)$, resp.) Propositin \ref{incoc} above implies the following fact.

\begin{cor}\label{C.Duali} Let $\omega$ be a non-quasianalytic weight function and $1\leq p<\infty$. Then the following properties are satisfied.
	\begin{enumerate}
		\item $\cO'_{M,\omega}(\R^N)=\cO'_{M,\omega,p}(\R^N)$ algebraically and topologically.
		\item $\cO'_{C,\omega}(\R^N)=\cO'_{C,\omega,p}(\R^N)$ algebraically and topologically.
	\end{enumerate}
	\end{cor}

Thanks to Proposition \ref{incoc} and Corollary \ref{C.Duali}, it is easy to show some structure theorems for the dual spaces $\cO'_{M,\omega}(\R^N)$  and $\cO'_{C,\omega}(\R^N)$ of  $\cO_{M,\omega}(\R^N)$ and $\cO_{C,\omega}(\R^N)$, respectively. To see this, we first give
a structure theorem for the space $(\mathcal{O}^m_{n,\omega,p}(\mathbb{R}^N), r_{m,n,p})'$. So,
 we  introduce the following weighted spaces.
 
\begin{defn}\label{Lpweighted}
	Let $\omega$ be a non-quasianalytic weight function. 
	
	(a) For  $1\leq p< \infty$ and  for $m\in\N$ and $n\in\Z$ we define  $\left(\oplus L^p(\mathbb{R}^N,\exp(n\omega(x))\, dx)\right)_{\omega,m,p}$  ($\left(\oplus L^p(\mathbb{R}^N,\exp(n\omega(x))\, dx)\right)_{\omega,-m,p}$, resp.) as the set of all sequences $\{f_\alpha\}_{\alpha\in\N_0^N}$ of Lebesgue measurable functions on $\R^N$ satisfying the following condition:
	\begin{align}\label{eq.spazioplus}
&	|\{f_\alpha\}_{\alpha\in\N_0^N}|^p_{m,n\omega,p}:=\sum_{\alpha\in\mathbb{N}^N_0} \|\exp(n\omega)f_\alpha \|_p^p \exp\left(pm\varphi^*_\omega\left(\frac{|\alpha|}{m}\right)\right)<\infty \nonumber \\
&	\left(|\{f_\alpha\}_{\alpha\in\N_0^N}|^p_{-m,n\omega,p}:= \sum_{\alpha\in\mathbb{N}^N_0} \|\exp(n\omega)f_\alpha \|_p^p \exp\left(-pm\varphi^*_\omega\left(\frac{|\alpha|}{m}\right)\right)<\infty, \ {\rm \ resp.}\right).
	\end{align}
	
	(b) For $m\in\N$ and $n\in\Z$ we define the space $\left(\oplus L^\infty(\mathbb{R}^N,\exp(n\omega(x)) dx)\,\right)_{\omega,m,\infty}$ ($\left(\oplus L^\infty(\mathbb{R}^N,\exp(n\omega(x)) dx)\,\right)_{\omega,-m,\infty}$, resp.) as the set of all the sequences $\{f_\alpha\}_{\alpha\in\N_0^N} $ of Lebesgue measurable functions on $\R^N$ satisfying the following condition: 
	\begin{align}\label{eq.spazioO}
	&|\{f_\alpha\}_{\alpha\in\N_0^N}|_{m,n\omega,\infty}:=\underset{\alpha\in\mathbb{N}^N_0}{\sup}\, \|\exp(n\omega)f_\alpha \|_\infty \exp\left(m\varphi^*_\omega\left(\frac{|\alpha|}{m}\right)\right)<\infty\nonumber\\
	& \left(|\{f_\alpha\}_{\alpha\in\N_0^N}|_{-m,n\omega,\infty}:= \underset{\alpha\in\mathbb{N}^N_0}{\sup}\, \|\exp(n\omega)f_\alpha \|_\infty \exp\left(-m\varphi^*_\omega\left(\frac{|\alpha|}{m}\right)\right)<\infty, \ {\rm \ resp.}\right).
	\end{align}
\end{defn}

\begin{rem}\label{R.reflSpaziSomma} It is straightforward to show that
for every  $1\leq p\leq\infty$, the space $\left(\left(\oplus L^p(\mathbb{R}^N,\exp(n\omega(x))\, dx)\right)_{\omega,m,p}, |\cdot|_{m,n\omega,p}\right)$ is a Banach space with strong dual given by  $\left(\left(\oplus L^{p'}(\mathbb{R}^N,\exp(-n\omega(x))\, dx)\right)_{\omega,-m,p'}, |\cdot|_{-m,-n\omega,p'}\right)$  for $1\leq p<\infty$ , $p'$ being the conjugate exponent of $p$. Hence,  the space $\left(\left(\oplus L^p(\mathbb{R}^N,\exp(n\omega(x))\, dx)\right)_{\omega,m,p}, |\cdot|_{m,n\omega,p}\right)$ is a reflexive Banach space  for $1<p<\infty$.
\end{rem}


\begin{thm}[Structure theorem for $(\mathcal{O}^m_{n,\omega,p}(\mathbb{R}^N), r_{m,n,p})'$]\label{struomnp}
	Let $\omega$ be a non-quasianalytic weight function, $T\in \mathcal{D}'_{\omega}(\mathbb{R}^N)$, $m\in\N$ and $n\in\Z$ and $1\leq p< \infty$. Then $T\in(\mathcal{O}^m_{n,\omega,p}(\mathbb{R}^N),r_{m,n,p})'$ if and only if 
	\begin{equation}\label{cdsomnp}
	T=\sum_{\alpha\in\mathbb{N}^N_0} \partial^\alpha f_\alpha
	\end{equation}
	with $\{f_\alpha\}_{\alpha\in\N_0^N} \subset \left(\oplus L^{p'}(\mathbb{R}^N,\exp(n\omega(x))\,dx)\right)_{\omega,m,p'}$, $p'$ being the exponent conjugate of $p$.
\end{thm}
\begin{proof} We treat only the case $1<p<\infty$. The case $p=1$ follows in a similar way.
	
	Let $T\in \mathcal{D}'_{\omega}(\mathbb{R}^N)$ be of the form $(\ref{cdsomnp})$.  We claim that  the linear functional defined by
	\begin{equation*}
	\langle T,\phi \rangle:=\sum_{\alpha\in\mathbb{N}^N_0} (-1)^{|\alpha|}\langle f_\alpha, \partial^\alpha \phi \rangle, \quad  \phi\in \mathcal{O}^m_{n,\omega,p}(\mathbb{R}^N),
	\end{equation*}
	belongs to $(\mathcal{O}^m_{n,\omega,p}(\mathbb{R}^N),r_{m,n,p})'$.
	Indeed, in view of H\"older inequality, we obtain for every $\phi\in \mathcal{O}^m_{n,\omega,p}(\mathbb{R}^N)$ that
	\begin{align*}
	|\langle T,\phi \rangle|&=\sum_{\alpha\in\mathbb{N}^N_0} |\langle f_\alpha, \partial^\alpha \phi \rangle| \\& \leq \sum_{\alpha\in\mathbb{N}^N_0} \|\exp(n\omega)f_\alpha\|_{p'} \|\exp(-n\omega)\partial^\alpha \phi\|_{p} \exp\left(-m\varphi^*_\omega\left(\frac{|\alpha|}{m}\right)+m\varphi^*_\omega\left(\frac{|\alpha|}{m}\right)\right)\\&\leq \left(\sum_{\alpha\in\mathbb{N}^N_0} \|\exp(n\omega)f_\alpha\|_{p'}^{p'} \exp\left(p'm\varphi_\omega^*\left(\frac{|\alpha|}{m}\right)\right)\right)^{\frac{1}{p'}}\times \\&\times\left(\sum_{\alpha\in\mathbb{N}^N_0} \|\exp(-n\omega)\partial^\alpha \phi\|_{p}^{p} \exp\left(-pm\varphi^*_\omega\left(\frac{|\alpha|}{m}\right)\right)\right)^{\frac{1}{p}}\\
	&=|\{f_\alpha\}_{\alpha\in\N_0^N}|_{m,n\omega,p'} r_{m,n,p}(\phi).
	\end{align*}
 This shows that $(\mathcal{O}^m_{n,\omega,p}(\mathbb{R}^N), r_{m,n,p})'$. So, the claim is proved.

	Conversely, suppose that $T\in(\mathcal{O}^m_{n,\omega,p}(\mathbb{R}^N), r_{m,n,p})'$. Then there exists a constant $C>0$ so that
	\begin{equation*}
	|\langle T,\phi \rangle|\leq Cr_{m,n,p}(\phi)=C\left(\sum_{\alpha\in\mathbb{N}^N_0}  \|\exp(-n\omega)\partial^\alpha \phi \|_{p}^{p} \exp\left(-mp\varphi^*_\omega\left(\frac{|\alpha|}{m}\right)\right)\right)^{\frac{1}{p}}
	\end{equation*}
	for each $\phi \in \mathcal{O}^m_{n,\omega,p}(\mathbb{R}^N)$. We now observe
 that the linear operator 
	\begin{equation*}
	J: \, (\mathcal{O}^m_{n,\omega,p}(\mathbb{R}^N), \, r_{m,n,p})\to \left(\left(\oplus L^{p}(\mathbb{R}^N,\exp(-n\omega(x))dx)\right)_{\omega,-m,p},  |\cdot|_{-m,-n\omega,p}\right)
	\end{equation*}
	defined by 
	$J(\phi):=  \{(-1)^{|\alpha|}\partial^\alpha \phi\}_ {\alpha\in\N_0^N}$ for $\phi\in \mathcal{O}^m_{n,\omega,p}(\mathbb{R}^N)$, 
	is an isometry and so a one-to-one operator, as
	\begin{equation*}
|J(\phi)|_{-m,-n\omega,p}=\left(\sum_{\alpha\in\mathbb{N}^N_0}\, \|\exp(-n\omega)\partial^\alpha \phi \|_{p}^{p} \exp\left(-pm\varphi^*_\omega\left(\frac{|\alpha|}{m}\right)\right)\right)^{\frac{1}{p}}=r_{m,n,p}(\phi).
	\end{equation*}
	Let $G:=J(\mathcal{O}^m_{n,\omega,p}(\mathbb{R}^N))$ and define on $G$ the linear functional  $\langle F,\{\phi_\alpha\}_{\alpha\in\N_0^N}\rangle:=\langle T, J^{-1}(\{\phi_\alpha\}_{\alpha\in\N_0^N})\rangle$ for  $\phi \in G$. Obviously, $F\in (G,|\cdot|_{-m,-n\omega,p})'$. In particular, we have for every $\{\phi_\alpha\}_{\alpha\in\N_0^N}\in G$ that 
	\begin{align*}
	|\langle F, \{\phi_\alpha\}_{\alpha\in\N_0^N}\rangle|&=|\langle T, J^{-1}(\{\phi_\alpha\}_{\alpha\in\N_0^N})\rangle|\leq Cr_{m,n,p}(J^{-1}(\{\phi_\alpha\}_{\alpha\in\N_0^N}))\\&=C|\{\phi_\alpha\}_{\alpha\in\N_0^n}|_{-m,-n\omega,p}.
	\end{align*}
	Thanks to the Hahn-Banach theorem, $F$ admits a continuous linear extension on $\left(\oplus L^{p}(\mathbb{R}^N,\exp(-n\omega(x)))\right)_{\omega,m,p}$ with the same norm. Denoting such an extension by $\tilde{F}$, we obtain that $\tilde{F}\in \left(\oplus L^{p'}(\mathbb{R}^N,\exp(n\omega(x))dx) \right)_{\omega,m,p'}$.
	This means  that there exists  $\{f_\alpha\}_{\alpha\in\N_0^N} \in \left(\oplus L^{p'}(\mathbb{R}^N,\exp(n\omega(x))dx)\right)_{\omega,m,p'}$ such that  $\tilde{F}= \sum_{\alpha\in\mathbb{N}^N_0} f_\alpha$, i.e., 
	\begin{equation*}
	\langle \tilde{F},\{\phi_\alpha\}_{\alpha\in\N_0^N} \rangle= \sum_{\alpha\in\mathbb{N}^N_0}\langle f_\alpha,\phi_\alpha\rangle= \sum_{\alpha\in\mathbb{N}^N_0}\int_{\mathbb{R}^N} f_\alpha(x) \phi_\alpha (x)\,dx
	\end{equation*}
	for every $\{\phi_\alpha\}_{\alpha\in\N_0^N}\in \left(\oplus L^{p}(\mathbb{R}^N,\exp(-n\omega(x))dx)\right)_{\omega,-m,p}$. So, for  $\phi \in \mathcal{O}^m_{n,\omega,p}(\mathbb{R}^N)$ we have
	\begin{align*}
	\langle T,\phi \rangle&= \langle F, \{(-1)^{|\alpha|}\partial^\alpha \phi\}_ {\alpha\in\N_0^N}\rangle=\langle \tilde{F}, \{(-1)^{|\alpha|}\partial^\alpha \phi\}_ {\alpha\in\N_0^N}\rangle\\
	&=  \sum_{\alpha\in\mathbb{N}^N_0}\int_{\mathbb{R}^N} (-1)^{|\alpha|}f_\alpha(x) \partial^\alpha \phi(x) dx= \sum_{\alpha\in\mathbb{N}^N_0}\int_{\mathbb{R}^N} \partial^\alpha f_\alpha(x) \phi(x) dx=\sum_{\alpha\in\mathbb{N}^N_0} \langle \partial^\alpha f_\alpha, \phi \rangle.
	\end{align*}
\end{proof}

As application of Theorem \ref{struomnp} and Proposition \ref{incoc}, we deduce  the following representations.


\begin{thm}[Structure theorem for $\mathcal{O}'_{M,\omega}(\mathbb{R}^N)$]\label{struom}
	Let $\omega$ be a non-quasianalytic weight function, $T\in \mathcal{D}'_{\omega}(\mathbb{R}^N)$ and $1\leq p<\infty$. Then $T\in \mathcal{O}'_{M,\omega}(\mathbb{R}^N)$ if and only if there exists $m\in\mathbb{N}$ such that for every  $n\in\mathbb{N}$ there exists  $\{f_\alpha\}_{\alpha\in\N_0^N} \subset \left(\oplus L^{p'}(\mathbb{R}^N,\exp(n\omega(x))\,dx)\right)_{\omega,m,p'}$, $p'$ being the conjugate exponent of $p$, such that 
	\begin{equation}
	T=\sum_{\alpha\in\mathbb{N}^N_0} \partial^\alpha f_\alpha.
	\end{equation}
\end{thm}

\begin{thm}[Structure theorem for $\mathcal{O}'_{C,\omega}(\mathbb{R}^N)$]\label{struoc}
	Let $\omega$ be a non-quasianalytic weight function, $T\in \mathcal{D}'_{\omega}(\mathbb{R}^N)$ and $1\leq p<\infty$. Then $T\in \mathcal{O}'_{C,\omega}(\mathbb{R}^N)$ if and only if for each $n\in\mathbb{N}$ there exist $m\in\mathbb{N}$ and  $\{f_\alpha\}_{\alpha\in\N_0^N} \subset \left(\oplus L^{p'}(\mathbb{R}^N,\exp(n\omega(x))\,dx)\right)_{\omega,m,p'}$, $p'$ being the conjugate exponent of $p$,  such that
	\begin{equation}
	T=\sum_{\alpha\in\mathbb{N}^N_0} \partial^\alpha f_\alpha.
	\end{equation}
\end{thm}

Finally, Theorem \ref{struomnp} combined with Remark \ref{Omnp.rappS} also implies the following representation.

\begin{thm}[Structure theorem for $\mathcal{S}'_{\omega}(\mathbb{R}^N)$]\label{stru}
	Let $\omega$ be a non-quasianalytic weight function, $T\in \mathcal{D}'_{\omega}(\mathbb{R}^N)$ and $1\leq p< \infty$. Then $T\in \mathcal{S}'_{\omega}(\mathbb{R}^N)$ if and only if there exist $m\in\mathbb{N}$  and with $\{f_\alpha\}_{\alpha\in\N_0^N} \subset \left(\oplus L^{p'}(\mathbb{R}^N,\exp(-m\omega(x))\,dx)\right)_{\omega,m,p'}$ such that
	\begin{equation}\label{cds}
	T=\sum_{\alpha\in\mathbb{N}^N_0} \partial^\alpha f_\alpha.
	\end{equation}
\end{thm}

The structure theorems above should be compared with \cite[Theorem 3.2]{DPPV} and with \cite[Theorem 3.2]{Kov} and \cite[Theorem 2]{Kov-2} holding in the frame of ultradistributions as introduced by Komatsu \cite{K}.

\begin{rem}\label{R.rifl} Let $\omega$ be a non-quasianalytic weight function.
	
	(a) From the proof of Theorem \ref{struomnp} it follows that the space $(O_{n,\omega,p}^m, r_{m,n,p})$ is a closed subspace of $\left(\left(\oplus L^p(\mathbb{R}^N,\exp(n\omega(x))\, dx)\right)_{\omega,m,p}, |\cdot|_{m,n\omega,p}\right)$. By Remark \ref{R.reflSpaziSomma}  we can then conclude that for $1<p<\infty$,  the Banach space $(O_{n,\omega,p}^m, r_{m,n,p})$ is reflexive.
	
	(b) Let $1<p<\infty$. Then by point (a) above it follows that for every $m\in\N$ the space $\ind_{\stackrel{\rightarrow}{n}}\cO^{m}_{n,\omega,p}(\R^N)$ is a complete reflexive (LB)-space. Hence, taking in account  Proposition \ref{incoc}(1) and \cite[Theorem 5.3]{De}, the space  $\cO_{M,\omega,p}(\R^N)$ is a complete reflexive lc-space. 
	
	(c) Let $1<p<\infty$. Then by point (a) above it follows that for every $n\in\N$ the space $\proj_{\stackrel{\leftarrow}{m}}\cO^{m}_{n,\omega,p}(\R^N)$ is a reflexive Fr\'echet space. Hence, $\cO_{C,\omega,p}(\R^N)$ (and hence, $\cO_{C,\omega}(\R^N)$ by Proposition \ref{incoc}) is an inductive limit of  reflexive Fr\'echet spaces.
\end{rem}

\section{The space $\cD_{L^p_\mu,\omega}(\mathbb{R}^N)$}

In  this section we collect further necessary results in order to show that $\cO'_{C,\omega}(\R^N)$ is the space of convolutors of $\cS_{\omega}(\R^N)$ and of $\cS'_{\omega}(\R^N)$. To this end we   introduce the spaces  $\cD_{L^p_\mu,\omega}(\mathbb{R}^N)$, where $\mu\in\R$ and $1\leq p\leq \infty$. 

\begin{defn}\label{D.spaziLp}
	Let $\omega$ be a non-quasianalytic weight function and let $\mu\in\R$. 
	
	(a)  We denote by $\cD_{L^p_\mu,\omega}(\mathbb{R}^N)$, for $1\leq p<\infty$,  the set of all functions $f\in {C}^\infty(\mathbb{R}^N)$ such that for every $m\in \mathbb{N}$ the following inequality is satisfied:
	\begin{equation}\label{dlp}
	t^p_{m,\mu,p}(f):=\sum_{\alpha\in \N^N_0}\,\|\exp(\mu\omega)\partial^\alpha f\|_p^p\exp \left(-pm\varphi^*_\omega\left(\frac{|\alpha|}{m}\right)\right)<\infty.		
	\end{equation}
	(b)  We denote by $\cB_{L^\infty_\mu,\omega}(\mathbb{R}^N)$  the set of all functions $f\in {C}^\infty(\mathbb{R}^N)$ such that for every $m\in \mathbb{N}$ the following inequality is satisfied:
	\begin{equation}\label{dlinfty}
	t_{m,\mu,\infty}(f):=\underset{\alpha\in\N^N_0}{\sup}\,\|\exp(\mu\omega)\partial^\alpha f\|_\infty\exp \left(-m\varphi^*_\omega\left(\frac{|\alpha|}{m}\right)\right)<\infty.		
	\end{equation}
	
	(c) We denote by $\cD_{L^{\infty}_{\mu},\omega}(\R^N)$ the subspace of $\cB_{L^\infty_\mu,\omega}(\mathbb{R}^N)$ consisting of all the functions $f$ such that $|\exp(\mu\omega(x))\partial^\alpha f(x)|\to 0$ as $|x|\to \infty$ for all $\alpha\in\N_0^N$.
	
	We denote by $\cD'_{L^p_\mu,\omega}(\mathbb{R}^N)$ (by $\cB'_{L^\infty_\mu,\omega}(\R^N)$, resp.) the strong dual of $\cD_{L^p_\mu,\omega}(\mathbb{R}^N)$ (of $\cB_{L^\infty_\mu,\omega}(\R^N)$, resp.).
	\end{defn}

The spaces $\cD_{L^p_\mu,\omega}(\R^N)$,  $
1\leq p\leq \infty$, and $\cB_{L^\infty_\mu,\omega}(\R^N)$ are always supposed to be endowed with the lc-topology generated by the sequence of norms $\{t_{m,\mu,p}\}_{m\in\N}$ and $\{t_{m,\mu,\infty}\}_{m\in\N}$, respectively. The elements
of the strong  dual $\cD'_{L^p_\mu,\omega}(\R^N)$ of $\cD_{L^p_\mu,\omega}(\R^N)$ are called ultradistributions of $L^{p'}_{\mu}$-growth, $p'$  being the conjugate exponent
 of $p$.
In the case case $\mu=0$ such spaces were already considered in \cite{BE}, \cite{Ci} and \cite{GO} and  are extensions of the classical spaces   $\cD_{L^p}(\R^N)$ and $\cB(\R^N)$ as introduced by Schwartz \cite{S} (see also \cite{BN}). 
Moreover, spaces analogous to $\cD_{L^p_\mu,\omega}(\R^N)$,  $
1\leq p\leq \infty$, (to $\cB_{L^\infty_\mu,\omega}(\R^N)$, resp.) were treated in \cite{Kov-2} (in \cite{De2}, resp.),  
in the context of ultradifferentiable functions of Beurling type as introduced by Komatsu \cite{K} (of Gelfand-Shilov type spaces, resp.).

\begin{rem}\label{R.Coincidenza} Since $t_{m,\mu,p}=\sigma_{m,\mu,p}$ for $\mu>0$ and $t_{m,-n,p}=r_{m,n,p}$ for $n\in\N\,(t_{m,-n,\infty}=r_{m,n}$ for $n\in\N$), the following  equalities are valid algebraically and topologically:
	\begin{equation}\label{eq.SD}
	\cS_{\omega}(\R^N)=\bigcap_{\mu>0} \cD_{L^p_{\mu},\omega}(\mathbb{R}^N)=\bigcap_{\mu>0}\cB_{L^\infty_{\mu},\omega}(\R^N)\ \ (1\leq p\leq\infty),
	\end{equation}
	\begin{equation}\label{eq.ODp}
	\bigcap_{m=1}^\infty \cO_{n,\omega,p}^m(\R^N)=\cD_{L^p_{-n},\omega}(\mathbb{R}^N) \ \   (1\leq p<\infty,\ n\in\N),
	\end{equation}
 	\begin{equation}\label{eq.OD}
\bigcap_{m=1}^\infty \cO_{n,\omega}^m(\R^N)=\cB_{L^\infty_{-n},\omega}(\R^N) \ \ 	  (n\in\N).
	\end{equation}
	and hence,
	\begin{equation}\label{eq.OCD}
	\cO_{C,\omega}(\R^N)=\bigcup_{n=1}^\infty\cD_{L^p_{-n},\omega}(\mathbb{R}^N) =\bigcup_{n=1}^\infty\cB_{L^\infty_{-n},\omega}(\mathbb{R}^N)\ \ (1\leq p\leq\infty),
	\end{equation}
	where \eqref{eq.OCD} has been established in Proposition \ref{incoc} for $1\leq p<\infty$. For $p=\infty$, \eqref{eq.OCD} follows from the fact that if $f\in \cB_{L^\infty_{-n},\omega}(\mathbb{R}^N)$ for some $n\in\N$, then $f\in \cD_{L^\infty_{-(n+1)},\omega}(\R^N)$, as it is easy to see.
\end{rem} 
In the following, we study the  convolution operators acting on the spaces of ultradistributions of $L^p_\mu$-growth, thereby extending some results in \cite{BE} (see also \cite{GO}) in the setting of the weigthed spaces $\cD_{L^p_\mu, \omega}(\R^N)$. We begin by collecting some useful properties about them.
 
 \begin{prop}\label{propdlpn}
	Let $\omega$ be a non-quasianalytic weight function and $1\leq p\leq \infty$. Then the following properties are satisfied.
	\begin{enumerate}
		\item For every $\mu\in\R$ the spaces $\cD_{L^p_\mu,\omega}(\mathbb{R}^N)$ and $\cB_{L^\infty_\mu,\omega}(\R^N)$ are    Fr\'echet spaces.
		\item For every $\mu,\mu'\in \R$ with $\mu< \mu'$ the inclusions  $\cD_{L^p_{\mu'},\omega}(\mathbb{R}^N)\hookrightarrow \cD_{L^p_{\mu},\omega}(\mathbb{R}^N)$ and $\cB_{L^\infty_{\mu'},\omega}(\mathbb{R}^N)\hookrightarrow \cB_{L^\infty_{\mu},\omega}(\mathbb{R}^N)$ are continuous.
		\item  For every $\mu\in\R$ the inclusion $\cD_{\omega}(\mathbb{R}^N)\hookrightarrow \cD_{L^p_\mu,\omega}(\R^N)$ is continuous with dense range.
		\item  For every $\mu\in\R$ the inclusions $\cD_{L^p_\mu,\omega}(\mathbb{R}^N)\hookrightarrow \cE_{\omega}(\mathbb{R}^N)$ and $\cB_{L^\infty_\mu,\omega}(\mathbb{R}^N)\hookrightarrow \cE_{\omega}(\mathbb{R}^N)$ are continuous with dense range.
	\end{enumerate}
\end{prop}
\begin{proof}
	The proof of assertions (1) and (2) are straightforward. 
	
	(3) Fix $\mu\in\R$. We first  observe that  the inclusion $\cD_{\omega}(\mathbb{R}^N)\hookrightarrow \cD_{L^p_\mu,\omega}(\R^N)$ is  clearly continuous. Indeed,  it is the composition of the continuous inclusions $\cD_{\omega}(\mathbb{R}^N)\hookrightarrow \cS_\omega(\R^N)$ and $\cS_\omega(\R^N)\hookrightarrow \cD_{L^p_\mu,\omega}(\R^N)$ (see Remark \ref{R.Coincidenza}). 
	We now assume  $p<\infty$ and  observe that the sequence $\{q_{m,\mu,p}\}_{m\in\N}$ of  norms also generates the lc-topolgy of $\cD_{L^p_\mu,\omega}(\R^N)$ in view of Remark \ref{R.Normenegativo}. 
	
	So, fix $f\in \mathcal{D}_{L^p_\mu,\omega}(\mathbb{R}^N)$
	and $\phi\in \mathcal{D}_{\omega}(\mathbb{R}^N)$ such that $\phi \equiv 1$ on $\overline{B}_1(0)$ and $0\leq \phi\leq 1$. Then for every $\epsilon >0$ the function $\phi_\epsilon (x):=\phi(\epsilon x)f(x) $, for $x\in\R^N$, belongs to  $ \mathcal{D}_{\omega}(\mathbb{R}^N)$ by property (3) above. Fix $m\in\mathbb{N}$ and let $M\in\mathbb{N}$ such that $M\geq Lm$, where $L\geq 1$ is the costant appearing in formula $(\ref{l})$. Then for every $\alpha\in\mathbb{N}^N_0$, $x\in\mathbb{R}^N$ and $\epsilon>0$ we have
	\begin{align*}\label{eq.densC}
	&||\exp(\mu\omega)\partial^\alpha(\phi_\epsilon -f)||_p\leq \left(\int_{\R^N}\exp(\mu p\omega(x))|\partial^\alpha f(x)(\phi(\epsilon x)-1)|^p\,dx\right)^{\frac{1}{p}}\\
	&\quad +\sum_{\beta<\alpha} \binom{\alpha}{\beta} \left(\int_{\R^N}\exp(\mu p\omega(x))|\partial^\beta f(x)|^p\epsilon^{(\alpha-\beta)p}|\partial^{\alpha-\beta} \phi (\epsilon x)|^p\, dx\right)^{\frac{1}{p}}\nonumber\\
	& \quad\leq \left(\int_{\R^N}\exp(\mu p\omega(x))|\partial^\alpha f(x)(\phi(\epsilon x)-1)|^p dx\right)^{\frac{1}{p}}\\&	\quad +\sum_{\beta<\alpha} \binom{\alpha}{\beta} \epsilon^{(\alpha-\beta)}q_{M,\mu,p}(f) \exp\left(M\varphi^*_\omega\left(\frac{|\beta|}{M}\right)\right) q_{M,0}(\phi)\exp\left(M\varphi^*_\omega\left(\frac{|\alpha-\beta|}{M}\right)\right)\nonumber\\
	&\quad \leq\left(\int_{\R^N}\exp(\mu p\omega(x))|\partial^\alpha f(x)(\phi(\epsilon x)-1)|^p dx\right)^{\frac{1}{p}}\\&\quad +\sum_{\beta<\alpha} \binom{\alpha}{\beta} \epsilon^{(\alpha-\beta)}q_{M,\mu,p}(f)q_{M,0}(\phi)\exp\left(M\varphi^*_\omega\left(\frac{|\alpha|}{M}\right)\right) \\
	&\quad \leq\left(\int_{\R^N}\exp(\mu p\omega(x))|\partial^\alpha f(x)(\phi(\epsilon x)-1)|^p dx\right)^{\frac{1}{p}}\\
	&\quad +q_{M,\mu,p}(f)q_{M,0}(\phi)\epsilon 2^{|\alpha|}\exp\left(M\varphi^*_\omega\left(\frac{|\alpha|}{M}\right)\right),
	\end{align*}
	after having observed  that  $\sum_{\beta<\alpha} \binom{\alpha}{\beta} \epsilon^{(\alpha-\beta)}\leq \epsilon 2^{|\alpha|}$ for every $\alpha\in\N_0^N$ and applied formula (\ref{secondprop}). Since $M\geq mL$, from $(\ref{firstprop})$ it follows  for every $\alpha\in\N_0^N$, $x\in\R^N$  and $\epsilon>0$ that 
	\begin{align*}
	\|\exp(\mu\omega)\partial^\alpha (\phi_\epsilon-f)\|_p
	& \leq\left(\int_{\R^N}\exp(\mu p\omega(x))|\partial^\alpha f(x)(\phi(\epsilon x)-1)|^p dx\right)^{\frac{1}{p}}\\
	& +q_{M,\mu,p}(f)q_{M,0}(\phi)\epsilon \exp(mL)\exp\left(m\varphi^*_\omega\left(\frac{|\alpha|}{m}\right)\right).
	\end{align*}
	Therefore, we have for every $\epsilon>0$ that 
	\begin{align*}
	&q_{m,\mu,p}(\phi_\epsilon-f)=\sup_{\alpha\in \N^N_0}\,\|\exp(\mu\omega)\partial^\alpha (\phi_\epsilon-f)\|_p\exp \left(-m\varphi^*_\omega\left(\frac{|\alpha|}{m}\right)\right)\\& \quad\leq \sup_{\alpha\in \N^N_0}\|\exp(\mu \omega)\partial^\alpha f(\phi(\epsilon \cdot)-1)\|_p\exp\left(-m\varphi^*_\omega\left(\frac{|\alpha|}{m}\right)\right) +\epsilon \exp(mL)q_{M,\mu,p}(f)q_{M,0}(\phi).
	\end{align*}
	It is clear that $\epsilon \exp(mL)q_{M,\mu,p}(f)q_{M,0}(\phi)\to0$, as $\epsilon \to 0^+$. So, it remains to prove that
	\begin{equation*}
	\sup_{\alpha\in \N^N_0}\|\exp(\mu \omega)\partial^\alpha f(\phi(\epsilon \cdot)-1)\|_p\exp\left(-m\varphi^*_\omega\left(\frac{|\alpha|}{m}\right)\right)\to 0
	\  \text{ as }\; \epsilon \to 0^+.
	\end{equation*}
	But,  $\phi(\epsilon x)-1=0$ whenever $|x|\leq \frac{1}{\epsilon}$. Accordingly,  we have
	\begin{align*}
	&\sup_{\alpha\in \N^N_0}\|\exp(\mu \omega)\partial^\alpha f(\phi(\epsilon \cdot)-1)\|_p\exp\left(-m\varphi^*_\omega\left(\frac{|\alpha|}{m}\right)\right)=
	\\&\quad=\sup_{\alpha\in \N^N_0}\,\exp\left(-m\varphi^*_\omega\left(\frac{|\alpha|}{m}\right)\right)\left(\int_{|x|> \frac{1}{\epsilon}}\exp(\mu p\omega(x))|\partial^\alpha f(x)(\phi(\epsilon x)-1)|^p dx\right)^{\frac{1}{p}}\\&\quad\leq \sup_{\alpha\in \N^N_0}\,\exp\left(-m\varphi^*_\omega\left(\frac{|\alpha|}{m}\right)\right)\left(\int_{|x|> \frac{1}{\epsilon}}\exp(\mu p\omega(x))|\partial^\alpha f(x)|^p\,dx\right)^{\frac{1}{p}}.
	\end{align*}
	Arguing in a similar way as in \cite[Lema 1.1.24]{GO-T}, we obtain that
	\begin{align*}
	\underset{\alpha\in\N^N_0}{\sup} \exp\left(-m\varphi^*_\omega\left(\frac{|\alpha|}{m}\right)\right)\left(\int_{|x|> \frac{1}{\epsilon}}\exp(\mu p\omega(x))|\partial^\alpha f(x)|^p \,dx\right)^{\frac{1}{p}} \to 0\, \textit{ as }\, \epsilon \to 0^+.
	\end{align*}
	By the arbitarity of $m\in\mathbb{N}$, we can conclude that $\phi_\epsilon \to f$ in $\cD_{L^p_\mu,\omega}(\R^N)$ as $\epsilon \to 0^+$.
	
	
	The proof of the density for $p=\infty$ is very similar to that of the case $p<\infty$. Indeed, fixed $f\in\cD_{L^\infty_\mu,\omega}(\R^N)$, we can show analogously that for the same $\phi_\epsilon$ and $M$ we have
	\begin{align*}
	&t_{m,\mu,\infty}(\phi_\epsilon-f)\leq \underset{\alpha\in\N^N_0}{\sup}\,\underset{|x|> \frac{1}{\epsilon}}{\sup}\,\exp\left(-m\varphi^*_\omega\left(\frac{|\alpha|}{m}\right)\right)\exp(\mu \omega(x))|\partial^\alpha f(x)|\\&\quad+\epsilon \exp(mL)t_{M,\mu,\infty}(f)q_{M,0}(\phi).
	\end{align*}
	The assumption $f\in\cD_{L^\infty_\mu,\omega}(\R^N)$ implies that 
	\begin{equation*}
	\underset{\alpha\in\N^N_0}{\sup}\,\underset{|x|> \frac{1}{\epsilon}}{\sup}\,\exp\left(-m\varphi^*_\omega\left(\frac{|\alpha|}{m}\right)\right)\exp(\mu \omega(x))|\partial^\alpha f(x)|\to 0 \,\textit{ as }\,\epsilon\to0^+
	\end{equation*}
	and so $t_{m,\mu,\infty}(\phi_\epsilon-f)\to 0$, as $\epsilon\to0^+$.
	
	(4) For a fixed $\mu\in\R$, let $n\in\N$ satisfying $-n\leq \mu$ and $n'\in\N$ with $n'\geq n$ as in Proposition \ref{PropInclp}. Then by point (2), Proposition \ref{incoc} (see   \eqref{eq.inLLp}) and Theorem \ref{T.incl} (cf. also with Remark \ref{R.Coincidenza}) it follows that the inclusions
	\[
	\cD_{L^p_{\mu},\omega}(\R^N)\hookrightarrow \cD_{L^p_{-n},\omega}(\R^N)=\bigcap_{m=1}^\infty\cO^m_{n,\omega,p}(\R^N)\hookrightarrow \bigcap_{m=1}^\infty\cO^m_{n',\omega}(\R^N)\hookrightarrow \cE_\omega(\R^N)
	\] 
	and the inclusions
	\[
	\cB_{L^\infty_{\mu },\omega}(\R^N)\hookrightarrow \cB_{L^\infty_{-n},\omega}(\R^N)=\bigcap_{m=1}^\infty \cO^{m}_{n,\omega}(\R^N)\hookrightarrow \cE_\omega(\R^N)
	\]
	are continuous.
	Finally the density of such inclusions follows from the fact that $\cD_\omega(\R^N)$ is continuously and densely included in   $\cE_\omega(\R^N)$.
\end{proof}

\begin{rem}
	The inclusion $\cD_{\omega}(\mathbb{R}^N)\hookrightarrow \cB_{L^\infty_\mu,\omega}(\R^N)$ is never dense. In particular, in point (3) of Proposition \ref{propdlpn}, we use the stronger condition $|\exp(\mu\omega(x))\partial^\alpha f(x)|\to 0$ as $|x|\to \infty$ for every $\alpha\in\N_0^N$ to get the claim.
\end{rem}

\begin{rem}\label{inlp} For every  $\mu\in \R$ and $1\leq p\leq \infty$ the space 
 $\cD_{L^p_{\mu},\omega}(\mathbb{R}^N)$ is clearly included in $ L^p(\R^N, \exp(\mu\omega(x))dx)$ and hence,  $\|\exp(\mu\omega)\partial^\alpha f\|_p<\infty$ for every $\alpha\in\N^N_0$ whenever  $f\in \cD_{L^p_{\mu},\omega}(\mathbb{R}^N)$.
\end{rem}

In the following, we always assume that the weight $\omega$ satisfies the additional condition $\log (1+t)=o(\omega(t))$ as $t\to\infty$, which is stronger than condition $(\gamma)$. 

Let  $G$ be an entire function satisfying the condition  $\log |G(z)| = O(\omega(z))$ as $|z|\to\infty$. Then the  functional $T_G$ defined on $\cE_\omega(\R^N)$  by
\begin{equation*}
\langle T_G,\phi\rangle:=\sum_{\alpha\in \N^N_0}(-i)^\alpha \frac{\partial^\alpha G(0)}{\alpha!}\partial^\alpha \phi(0), \quad \phi\in\cE_\omega(\R^N),
\end{equation*}
belongs to $\cE'_\omega(\R^N)$. The operator $G(D)$ defined on $\cD'_\omega(\R^N)$
through
\begin{equation*}
G(D): \cD'_\omega(\R^N)\to\cD'_\omega(\R^N), \quad S\mapsto G(D)S:= T_G\star S,
\end{equation*}
is called an ultradifferential operator of $\omega$-class. When $G(D)$ is restricted to $\mathcal{E}_\omega(\R^N)$, $G(D)$
is a continuous linear operator from $\mathcal{E}_\omega(\R^N)$ into itself and, for every $\phi\in\mathcal{E}_\omega(\R^N)$,
\begin{equation*}
(G(D)\phi)(x)=\sum_{\alpha\in\N^N_0}i^{|\alpha|} \frac{\partial^\alpha G(0)}{\alpha!}\partial^\alpha \phi(x), \quad \forall x\in\R^N.
\end{equation*}

Along the lines in \cite{GO} (see also \cite{BE,Braun}), the next aim  consists in  showing that each ultradifferential operator $G(D)$ of
$\omega$-class defines a continuous linear operator from $\cD_{L^p_\mu,\omega}(\mathbb{R}^N)$  into itself and  from $\cB_{L^\infty_\mu,\omega}(\R^N$) into itself, for every $1\leq p\leq \infty$ and $\mu\in\R$.

\begin{prop}\label{dlpngd}
	Let $\omega$ be a non-quasianalytic weight function  with the additional condition $\log (1+t)=o(\omega(t))$ as $t\to\infty$. Let $\mu\in\R$ and  $1\leq p\leq \infty$. If  $G(D)$ is an ultradifferential operator of $\omega$-class, then 
	\begin{equation*}
	G(D): \cD_{L^p_\mu,\omega}(\mathbb{R}^N) \to \cD_{L^p_\mu,\omega}(\mathbb{R}^N)
	\end{equation*}
	is a continuous linear operator. Moreover, $G(D)$ is also a continuous linear operator from $\cD'_{L^p_\mu,\omega}(\mathbb{R}^N)$ into itself. The result is also valid when $G(D)$ acts on  ${\cB}_{L^\infty_\mu,\omega}(\mathbb{R}^N)$ and on its strong dual, i.e., $G(D)\in \cL( \cB_{L^\infty_\mu,\omega}(\mathbb{R}^N))$ and $G(D)\in \cL( \cB'_{L^\infty_\mu,\omega}(\mathbb{R}^N))$. 
\end{prop}
\begin{proof}
	We treat only the case $p<\infty$. The proof of the other one is analogous and so it is omitted.
	
Since $G$ is entire function satisfying the condition $\log |G(z)| = O(\omega(|z|))$ as $|z|\to\infty$, by \cite[Lema 1.2.1]{GO-T} (or see the proof of \cite[Proposition 2.4]{GO} ), 
 there exists $k>0$ such that for every $\alpha\in\N_0^N$ we have
	\begin{equation*}
	|\partial^\alpha G(0)|\leq \alpha! \exp\left(k-k\varphi_\omega^*\left(\frac{|\alpha|}{k}\right)\right).
	\end{equation*}
Let	$f \in \cD_{L^p_\mu,\omega}(\mathbb{R}^N)$  and $h\in\N$. 
If  we choose 	
	 $m\in\mathbb{N}$ such that $m\geq\max\{kL,hL\}$, where $L\geq 1$ is the constant appearing in  (\ref{l}), it follows via (\ref{secondprop}) for every $\alpha,\beta\in \N^N_0$ that 
	\begin{align}\label{eq1}
	&\frac{|\partial^\alpha G(0)|^p}{(\alpha!)^p} \int_{\mathbb{R}^N}|\partial^{\alpha+\beta} f(x)|^p\exp(p\mu\omega(x))\, dx\leq \nonumber\\
	&\quad\leq t_{2m,\mu,p}^p(f) \exp \left(pm\varphi^*_\omega\left(\frac{|\alpha|}{m}\right)+pm\varphi^*_\omega\left(\frac{|\beta|}{m}\right)\right)\exp\left(pk-pk\varphi^*_\omega\left(\frac{|\alpha|}{k}\right)\right).
	\end{align}
	Since $\frac{\varphi^*_\omega(t)}{t}$ is an increasing function in $(0,\infty)$ and $\varphi^*_\omega$ satisfies inequality \eqref{eq.bb}, we get from (\ref{eq1}) that for every $\alpha,\beta\in \N^N_0$ 
	\begin{align*}
	&\frac{|\partial^\alpha G(0)|^p}{(\alpha!)^p}\exp \left(-ph\varphi^*_\omega\left(\frac{|\beta|}{h}\right)\right) \int_{\mathbb{R}^N}|\partial^{\alpha+\beta} f(x)|^p\exp(p\mu\omega(x))\, dx\nonumber\leq \\
	&\quad\leq t^p_{2m,\mu,p}(f)\exp(pk) \exp \left(pm\varphi^*_\omega\left(\frac{|\alpha|}{m}\right)-pk\varphi^*_\omega\left(\frac{|\alpha|}{k}\right)\right)\exp(phL-p|\beta|)\nonumber\\
	&\quad\leq t_{2m,\mu,p}^p(f)\exp(pk)\exp(pkL-p|\alpha|)\exp(phL-p|\beta|).
	\end{align*}
	Accordingly, we obtain for every $\beta\in\N_0^N$ that
	\begin{align*}
	& \exp \left(-ph\varphi^*_\omega\left(\frac{|\beta|}{h}\right)\right)||\exp(\mu \omega)\partial^\beta (G(D)f)||_p^p\leq\\
	&\quad\leq \left(\sum_{\alpha\in\N_0^N}\frac{|\partial^\alpha G(0)|}{\alpha!}\exp\left(-h\varphi^*_\omega\left(\frac{|\beta|}{h}\right)\right)||\exp(\mu\omega)\partial^{\alpha+\beta}f||_p\right)^p\\
	&\quad\leq t_{2m,\mu,p}^p(f)\exp(p(k+kL+hL))\exp(-p|\beta|) \left(\sum_{\alpha\in\N_0^N}\exp(-|\alpha|)\right)^p.
	\end{align*}
	This implies that
	\begin{align*}
	 t_{h,\mu,p}(G(D)f)&=\left(\sum_{\beta\in\N_0^N}||\exp(\mu \omega)\partial^\beta (G(D)f)||_p^p\exp \left(-ph\varphi^*_\omega\left(\frac{|\beta|}{h}\right)\right)\right)^{\frac{1}{p}} \\
&\leq Ct_{2m,\mu,p}(f)\left(\sum_{\beta\in\N^N_0}\exp(-p|\beta|) \right)^{\frac{1}{p}},
	\end{align*}
	where $\left(\sum_{\beta\in\N^N_0}\exp(-p|\beta|) \right)^{\frac{1}{p}}<\infty$ and $C:=\exp(k+kL+hL)\sum_{\alpha\in\N_0^N}\exp(-|\alpha|)<\infty$.  This shows that the map $G(D)$ is well posed and continuous.
	
	Since $G(-D):=\hat{G}(D)$ is also an ultradifferentiable  operator of $\omega$-class and hence $G(-D):  \cD_{L^p_\mu, \omega}(\R^N)\to \cD_{L^p_\mu, \omega}(\R^N)$ is a continuous linear operator as proved above, the dual operator $G(-D)': \cD'_{L^p_\mu, \omega}(\R^N)\to \cD'_{L^p_\mu, \omega}(\R^N)$ is a continuous linear operator too. Therefore, it follows for every $S\in \cD'_{L^p_\mu, \omega}(\R^N)$ and $f\in \cD_{L^p_\mu, \omega}(\R^N)$ that
	\[
	\langle G(-D)'S,f \rangle= \langle S,\check{T}_G \star f \rangle=\langle T_G\star S,  f \rangle=\langle G(D) S,  f \rangle.
	\]
	This implies that $G(D)$ is a continuous linear operator from $\cD'_{L^p_\mu,\omega}(\R^N)$ into itself.
\end{proof}

From  Proposition \ref{dlpngd} and Remark \ref{R.Coincidenza}, it easily follows the next result.

\begin{prop}\label{ocgd}
	Let $\omega$ be a non-quasianalytic weight function with the additional condition $\log(1+t)=o(\omega(t))$ as $t\to\infty$ and $G(D)$ be an ultradifferential operator of	$\omega$-class. Then the following properties are satisfied.
	\begin{enumerate}
		\item  $
		G(D): \cO_{C,\omega}(\mathbb{R}^N) \to \cO_{C,\omega}(\mathbb{R}^N)
	$
		is a continuous linear operator. Moreover, $G(D)$ also acts continuously on $\cO'_{C,\omega}(\mathbb{R}^N)$. 
		
		\item  $
		G(D): \cS_{\omega}(\mathbb{R}^N) \to \cS_\omega(\mathbb{R}^N)
		$
		is a continuous linear operator. Moreover, $G(D)$ also acts  continuously  on $\cS'_{\omega}(\mathbb{R}^N)$. 
	\end{enumerate}	
\end{prop}

\begin{proof} (1) Fix $p\in [1,\infty)$. Then by Proposition \ref{dlpngd} we have for every $n\in\N$ that $G(D)$ is a continuous linear operator from $\cD_{L^p_{-n},\omega}(\mathbb{R}^N)$ into itself. Since 
 $\cO_{C,\omega}(\mathbb{R}^N)=\bigcup_{n=1}^\infty \cD_{L^p_{-n},\omega}(\mathbb{R}^N)$ algebraically and topologically (cf. Remark \ref{R.Coincidenza}), it clearly follows the thesis. 
 
To show  that $G(D)\in \cL(\cO'_{C,\omega}(\R^N))$, it suffices to argue as in the final part of the proof of Proposition \ref{dlpngd}.  

The proof of (2) follows with similar arguments.\end{proof}


From now on in this section, we consider only the case $p\in [1,\infty)$. In particular, we give   representations of the elements of  $\cD'_{L^p_\mu,\omega}(\R^N)$ similar to the
ones in \cite{BE,GO}. In order to do this, we begin with the following result.

\begin{prop}\label{primaimp}
	Let $\omega$ be a non-quasianalytic weight function. 
	Let $\mu\in\R$ and $1\leq p< \infty$. Then there exists $\mu'\in\R$ with $\mu'\geq \mu$ such that  for every  $T\in \cD'_{L^p_\mu,\omega}(\mathbb{R}^N)$  and $\phi\in\cD_\omega(\R^N)$ the function $T\star\phi \in L^{p'}(\R^N,\exp(-\mu'\omega(x))dx)$,  $p'$ being the conjugate exponent of $p$. 
\end{prop}

\begin{proof}	
	Since $\cD'_{L^p_\mu,\omega}(\mathbb{R}^N)\subset \cD'_\omega(\R^N)$ by Proposition \ref{propdlpn}(2), the convolution is well posed and $T\star\phi\in \cE_\omega(\mathbb{R}^N)$ (see \cite[Proposition 6.4]{BMT}).
	By assumption there exist $m\in\N$ and  $C>0$ such that for every  $\varphi \in \cD_{L^p_\mu,\omega}(\mathbb{R}^N)$
	\begin{equation*}
	|\langle T,\varphi \rangle| \leq C t_{m,\mu,p}(\varphi).  
	\end{equation*}
	Therefore, we get  for every $\varphi \in \cD_{\omega}(\mathbb{R}^N)$ that 
	\begin{align*}
	\left|\int_{\mathbb{R}^N} (T\star \phi)(x)\varphi(x)\, dx\right|&=|\langle T\star\phi,\varphi\rangle|=|\langle T,\check{\phi}\star\varphi\rangle|\leq C t_{m,\mu,p}(\check{\phi}\star\varphi)\\
	&=C\left(\sum_{\alpha\in\N^N_0}\,\|\exp(\mu\omega)\partial^\alpha (\check{\phi}\star\varphi)\|_p^p\exp \left(-pm\varphi^*_\omega\left(\frac{|\alpha|}{m}\right)\right)\right)^{\frac{1}{p}}.
	\end{align*}
We now suppose that $\mu\geq 0$ and observe that by Remark \ref{inlp} and property $(\alpha)$ of $\omega$, we have for every $\varphi\in\cD_\omega(\R^N)$ that 
	\begin{align*}
	\|\exp(\mu\omega)\partial^\alpha (\check{\phi}\star\varphi)\|_p^p&=\int_{\mathbb{R}^N}\exp(p\mu\omega(x))|\partial^\alpha(\check{\phi}\star\varphi)(x)|^p\, dx\\
	&=\int_{\mathbb{R}^N}\exp(p\mu\omega(x))|(\partial^\alpha\check{\phi}\star\varphi)(x)|^p\, dx\\
	&=\int_{\mathbb{R}^N}\int_{\mathbb{R}^N}\exp(p\mu\omega(x-y+y))|\partial^\alpha \phi(x-y)|^p|\varphi(y)|^p\, dxdy\\
	&\leq \exp(Kp\mu)\|\exp(\mu K\omega)\varphi\|_p^p \|\exp(\mu K\omega)\partial^\alpha \phi\|_p^p,
	\end{align*}
	where $K$ is the constant appearing in condition $(\alpha)$. Thereby, we obtain for every $\varphi\in\cD_\omega(\R^N)$ that 
	\begin{align*}
	\left|\int_{\mathbb{R}^N} (T\star \phi)(x)\varphi(x)\, dx\right|&\leq C\exp(\mu K)\|\exp(\mu K\omega)\varphi\|_p t_{m,K\mu,p}(\phi)\\
	&=C'\|\exp(\mu K\omega)\varphi\|_p t_{m,K\mu,p}(\phi).
	\end{align*}
	This implies that $T\star\phi \in L^{p'}(\R^N,\exp(-\mu K\omega(x))dx)$ with $||\exp(-\mu K\omega)  T\star \phi||_{p'}\leq C' t_{m,K\mu,p}(\phi)$. In the case that $\mu<0$, a similar argument shows that 
$T\star\phi \in L^{p'}(\R^N,\exp(-\frac{\mu}{ K}\omega(x))dx)$ with $||\exp(-\frac{\mu} {K}\omega)  T\star \phi||_{p'}\leq C' t_{m,-\mu,p}(\phi)$.
 \end{proof}

\begin{rem}\label{R.SUB} In case the weight function $\omega$ is sub-additive, i.e., $\omega(s+t)\leq \omega(s)+\omega(t)$ for all $s,t\geq 0$, we can deduce from the proof above that for every $T\in\cD_{L^p_\mu, \omega }(\R^N)$ and $\phi\in \cD_\omega(\R^N)$ the function $T\star \phi\in L^{p'}(\R^N,\exp(-\mu\omega(x))dx)$, with $p'$ the conjugate exponent of $p$. 
	
\end{rem}



We recall that an ultradifferential operator $G(D)$ of $\omega$-class is said to be \textit{strongly elliptic} if there exist $M>0$ and $l>0$ such that $|G(z)|\geq M\exp(l\omega(z))$, for every $z\in \C^N$ with $|\Im z|<M|\Re z|$.

\begin{prop}\label{secondaimp}
	Let $\omega$ be a non-quasianalytic weight function with the additional condition $\log(1+t)=o(\omega(t))$ as $t\to \infty$. Let $T\in \cD'_\omega(\mathbb{R}^N)$, $1\leq p< \infty$ and $\mu\in\R$. Suppose that $T\star\phi \in L^{p}(\R^N,\exp(\mu\omega(x))dx)$ for every $\phi\in\cD_\omega(\R^N)$. Then, there exist a strongly elliptic ultradifferential operator $G(D)$ of $\omega$-class and $f,g\in L^{p}(\R^N,\exp(\mu\omega(x))dx)$ such that $T=G(D)f+g$.
\end{prop}

\begin{proof}
	Let $V_{p'}$ denote the unit ball of $L^{p'}(\mathbb{R}^N, \exp(-\mu\omega(x))dx)$, $p'$ being the conjugate exponent of $p$. Then,  for a fixed $\varphi \in V_{p'}\cap \cD_{\omega}(\R^N)$,  we have for every $\phi\in\cD_\omega(\R^N)$ that
	\begin{equation*}
	|\langle T\star \check{\varphi},\check{\phi}\rangle|=|\langle T\star \phi,{\varphi}\rangle|\leq \|\exp(\mu\omega)(T\star\phi)\|_p\|\exp(-\mu\omega)\varphi\|_{p'}\leq \|\exp(\mu\omega)(T\star\phi)\|_p.
	\end{equation*}
	This implies that  $\{T\star \check{\varphi}\, :\,\varphi \in V_{p'}\cap \cD_{\omega}(\R^N)  \}$ is a weakly bounded subset, and hence, an equicontinuous subset of $\cD'_\omega(\R^N)$. Therefore, if $K_1:=[-2,2]^N$, we can find $m\in \N$ and $C>0$ such that
	\begin{equation*}
	|\langle T\star \check{\varphi},\phi\rangle|\leq C p_{K_1,m}(\phi)
	\end{equation*}
	for each $\phi\in \cD_\omega(K_1)$ and $\varphi \in V_{p'}\cap \cD_{\omega}(\R^N)$. Accordingly, for each $\phi\in \cD_\omega(K_1)$ and $\varphi \in \cD_{\omega}(\R^N)$
	\begin{equation}\label{eq24}
	|\langle T\star \check{\varphi},\phi\rangle|\leq C p_{K_1,m}(\phi)\|\exp(-\mu\omega)\varphi\|_{p'}.
	\end{equation}
	We now take $K_2:=[-1,1]^N$ and we show that $T\star \phi \in L^{p}(\R^N,\exp(\mu\omega(x))dx)$ for every $\phi\in \cE_{\omega,2m}(K_2)\cap \cD(K_2)$. Let $\eta\in\cD_\omega(K_2)$ be such that $\eta\geq0$, $\int_{K_1}\exp(\mu\omega(x))\eta(x)\,dx=1$ and consider $\eta_\epsilon(x):=\frac{\eta(\frac{x}{\epsilon})}{\epsilon}$, for $x\in\R^N$ and $\epsilon>0$. Then, for $\phi\in \cE_{\omega,2m}(K_2)\cap \cD(K_2)$, $\phi\star\eta_\epsilon\in \cD_\omega(K_1)$, $0<\epsilon<1$, and $\phi\star\eta_\epsilon\to\phi$ in $\cE_{\omega,m}(K_1)\cap \cD(K_1)$ , as $\epsilon\to0^+$. By assumption, $T\star(\phi\star\eta_\epsilon)\in L^{p}(\R^N,\exp(\mu\omega(x))dx)$, $0<\epsilon<1$. On the other hand,    from (\ref{eq24}) it follows for every $0<\epsilon<1$ that
	\begin{equation}\label{eq25}
	\|\exp(\mu\omega)( T\star(\phi\star\eta_\epsilon)) \|_p\leq C p_{K_1,m}(\phi\star\eta_\epsilon)\leq Cp_{K_1,m}(\phi).
	\end{equation}
Thanks to inequality (\ref{eq25}),  we  get  that $\{T\star(\phi\star\eta_\epsilon)\}_{0<\epsilon<1}$ is a Cauchy net in the space $L^{p}(\R^N,\exp(\mu\omega(x))dx)$, thereby a convergent net in $L^{p}(\R^N,\exp(\mu\omega(x))dx)$. Since $T\in\cD'_\omega(\R^N)$, there exist $l\in\N$ and  $C'>0$ such that
	\begin{equation*}
	|\langle T,\varphi\rangle|\leq C'p_{K_2,l}(\phi)
	\end{equation*}
	for each $\phi\in\cD_\omega(K_2)$. Then $T$ can be continuously extended to $\cE_{\omega,l}(K_2)\cap \cD(K_2)$. Hence, if $m$ is large enough, we can conclude  for every $x\in\R^N$ that
	\begin{equation*}
	|(T\star(\phi\star\eta_\epsilon)-T\star\phi)(x)|\leq C'p_{K_2,l}(\phi\star\eta_\epsilon-\phi)\leq C'p_{K_1,m}(\phi\star\eta_\epsilon-\phi).
	\end{equation*}
	So, $T\star(\phi\star\eta_\epsilon)\to T\star\phi$ in $C_b(\R^N)$  as $\epsilon\to 0^+$. From (\ref{eq25}) we obtain that $T\star \phi \in L^{p}(\R^N,\exp(\mu\omega(x))dx)$. Now applying \cite[Corollary 2.6]{GO},
 we can write $\delta=G(D)\Gamma+\chi$, where $G(D)$ is a strongly elliptic ultradifferential operator of $\omega$-class, $\chi\in\cD_\omega(K_2)$ and $\Gamma\in\cE_{\omega,2m}(K_2)\cap \cD(K_2)$. To get the claim, is sufficient to take $f:=T\star\Gamma$ and $g:=T\star\chi$.
	\end{proof}

The results above lead to the following result about the elements of $\cD'_{L^p_\mu,\omega}(\R^N)$. 

\begin{thm}\label{cardlpdual}
	Let $\omega$ be a non-quasianalytic weight function with the additional condition $\log(1+t)=o(\omega(t))$ as $t\to\infty$ and  $T\in\cD'_\omega(\R^N)$. Let $1\leq p <\infty$ and $\mu\in\R$. Consider the following properties.
	\begin{enumerate}
	\item  $T\in \cD'_{L^p_\mu,\omega}(\R^N)$.
	\item There exists $\mu'\in\R$ with $\mu'\geq \mu$ such that $T\star\phi \in L^{p'}(\R^N,\exp(-\mu'\omega(x))dx)$ for every $\phi\in\cD_\omega(\R^N)$, $p'$ being the conjugate exponent of $p$.
	\item
	 There exist  $\mu'\in\R$ with $\mu'\geq \mu$, $G(D)$ a strongly elliptic ultradifferential operator of $\omega$-class and $f, g\in L^{p'}(\R^N,\exp(-\mu'\omega(x))dx)$ such that $T=G(D)f+g$, $p'$ being the conjugate exponent of $p$.
	 \end{enumerate}
 Then $(1)\Rightarrow (2)\Rightarrow (3)$. If, in addition, the weigth function $\omega$ is sub-additive, then all the assertions are equivalent.
\end{thm}

\begin{proof}
	$(1)\Rightarrow (2)$ follows from Proposition \ref{primaimp}.
	
	$(2)\Rightarrow (3)$ follows  from Proposition \ref{secondaimp}. 
	
	In case the weight function $\omega$ is sub-additive, by Remark \ref{R.SUB} property (1) implies property (2) with $\mu'=\mu$ and so,  property (2) implies property (3) with $\mu'=\mu$. Therefore, we can assume that $\mu'=\mu$ in (3). Now, it is clear that
	$(3)\Rightarrow (1)$. Indeed, from the facts that $L^{p'}(\R^N,\exp(-\mu\omega(x)dx))\subset \cD'_{L^p_\mu,\omega}(\R^N)$ and $G(D)\in \cL(\cD'_{L^p_\mu,\omega}(\R^N))$ it follows that $T=G(D)f+g\in  \cD'_{L^p_\mu,\omega}(\R^N)$.	
\end{proof}

Taking into account Remark \ref{R.Coincidenza} and applying Theorems \ref{cardlpdual} together with Proposition  \ref{ocgd}, we can obtain a second structure theorem for both the spaces $\cO'_{C,\omega}(\R^N)$ and $\cS'_\omega(\R^N)$. 

\begin{thm}[Second structure theorem for $\cO'_{C,\omega}(\R^N) $]
		Let $\omega$ be a non-quasianalytic weight function with the additional condition $\log(1+t)=o(\omega(t))$ as $t\to\infty$,  $T\in\cD'_\omega(\R^N)$ and $1\leq p<\infty$. Then $T\in \cO'_{C,\omega}(\R^N)$ if and only if for every $n\in\N$ there exist $n'(=[n/K])\in\N$ with $n'\leq n$, a strongly elliptic differential operator $G(D)$ of $\omega$-class and $f,g\in L^{p'}(\R^N, \exp(n'\omega(x))dx)$ such that $T=G(D)f+g$.

\end{thm}

\begin{thm}[Second structure theorem for $\cS'_{\omega}(\R^N) $]
	Let $\omega$ be a non-quasianalytic weight function with the additional condition $\log(1+t)=o(\omega(t))$ as $t\to\infty$,  $T\in\cD'_\omega(\R^N)$ and $1\leq p<\infty$. Then $T\in \cS'_{\omega}(\R^N)$ if and only if there exists $\mu>0$, a strongly elliptic differential operator $G(D)$ of $\omega$-class and $f,g\in L^{p'}(\R^N, \exp(-\mu\omega(x))dx)$ such that $T=G(D)f+g$.
	\end{thm}

\section{$\cO'_{C,\omega}(\R^N)$ is the space of convolutors of the spaces $\cS_\omega(\R^N)$ and $\cS'_\omega(\R^N)$ 
}
In this section we show that  $\mathcal{O}'_{C,\omega}(\mathbb{R}^N)$ is the space of convolutors of the spaces $\cS_\omega(\R^N)$ and $\cS'_\omega(\R^N)$. 
So, we begin by proving that the convolution beetween elements of  $\mathcal{S}'_{\omega}(\mathbb{R}^N)$ with the ones of  $\mathcal{S}_{\omega}(\mathbb{R}^N)$ belongs to $\mathcal{O}_{C,\omega}(\mathbb{R}^N)$. To this end, we observe the following facts.

\begin{lem}\label{cinf}
	Let $\omega$ be a non-quasianalytic weight function. If $T \in \mathcal{S}'_\omega(\mathbb{R}^N)$ and $f \in \mathcal{S}_\omega(\mathbb{R}^N)$, then the map
	\begin{equation*}
	\mathbb{R}^N \ni x \mapsto \langle T_y, \tau_xf\rangle
	\end{equation*}
	is a $C^\infty$ function in $x\in \mathbb{R}^N$. In particular, for every $\alpha \in \mathbb{N}^N_0$, we have 
	\begin{equation}\label{dp}
	\partial^\alpha_x \langle T_y,\tau_x f\rangle= \langle T_y, \partial^\alpha_x\tau_x f\rangle 
	\end{equation}
\end{lem}

We recall that the notation $T_y$ means that the distribution $T$ acts on a function $\psi(x-y)$, when the latter is regarded as a function of the variable $y$. 

\begin{proof}
	%
	Since the map $x\to \tau_x f$ is a $C^\infty$ function of $x\in \mathbb{R}^N$ with values in $\mathcal{S}_\omega(\mathbb{R}^N_y)$  as it is well-known, and  $T \in \mathcal{S}'_\omega(\mathbb{R}^N)$, we can apply \cite[Theorem 27.1]{Treves}  to conclude that the map $x \mapsto \langle T_y, \tau_xf\rangle$ is a $C^\infty$ function of $x\in \mathbb{R}^N$ that satisfies $(\ref{dp})$.
\end{proof}


\begin{prop}\label{conv}
	Let $\omega$ be a non-quasianalytic weight function. If $T \in \mathcal{S}'_\omega(\mathbb{R}^N)$ and $f \in \mathcal{S}_\omega(\mathbb{R}^N)$, then $T\star f \in \mathcal{O}_{C,\omega}(\mathbb{R}^N)$. Moreover the map $f\mapsto T \star f$ is a continuous linear operator from $\mathcal{S}_\omega(\mathbb{R}^N)$ into $\mathcal{O}_{C,\omega}(\mathbb{R}^N)$.
\end{prop}
\begin{proof}
	By Theorem \ref{stru} with $p=1$, there exist $m\in \mathbb{N}$ and a sequence $\{f_\alpha\}_{\alpha\in\N_0^N} \subset \left(\oplus L^\infty(\mathbb{R}^N,\exp(-m\omega(x))\,dx)\right)_{\omega,m,\infty}$ such that $T=\sum_{\alpha\in\mathbb{N}^N_0} \partial^\alpha f_\alpha$. Moreover, by Lemma \ref{cinf} the function $T \star f \in C^\infty(\mathbb{R}^N)$. In particular, we have for every $x\in\mathbb{R}^N$ and $\beta \in \mathbb{N}^N_0$ that 
	\begin{align*}
	\partial^\beta (T\star f)(x) & =(T\star \partial^\beta f)(x)=\sum_{\alpha\in\mathbb{N}^N_0} (\partial^\alpha f_\alpha \star \partial^\beta f)(x)\\
	&= \sum_{\alpha\in\mathbb{N}^N_0} \int_{\mathbb{R}^N} (-1)^{|\beta|} f_\alpha(y)\partial^{\alpha+\beta} f(x-y)\, dy
	\end{align*}
	and so
	\begin{align*}
	|\partial^\beta (T\star f)(x)| &\leq \sum_{\alpha\in\mathbb{N}^N_0} \int_{\mathbb{R}^N} |f_\alpha(y)\partial^{\alpha+\beta} f(x-y)|\, dy  \\ & \leq \sum_{\alpha\in\mathbb{N}^N_0} \|\exp(-m\omega)f_\alpha\|_\infty \|\exp(m\omega)\partial^{\alpha+\beta} \tau_x f\|_1.
	\end{align*}
	On the other hand, we have for every $x\in\R^N$ and $\alpha,\beta\in\N_0^N$ that
	\begin{align*}
&	\|\exp(m\omega)\partial^{\alpha+\beta} \tau_x f\|_1   =\int_{\mathbb{R}^N} \exp(m\omega(y))|\partial^{\alpha+\beta} f(x-y)|\, dy\\
&=\int_{\mathbb{R}^N} \exp(m\omega(x-z))|\partial^{\alpha+\beta} f(z)|\, dz \leq \int_{\mathbb{R}^N} \exp(mK(1+\omega(x)+\omega(z)))|\partial^{\alpha+\beta} f(z)|\, dz\\&=\exp(mK(1+\omega(x)))\int_{\mathbb{R}^N} \exp(mK\omega(z))|\partial^{\alpha+\beta} f(z)|\, dz\\&=\exp(mK(1+\omega(x)))\|\exp(mK\omega)\partial^{\alpha+\beta} f\|_1 \leq \exp(n(1+\omega(x)))\|\exp(n\omega)\partial^{\alpha+\beta} f\|_1,
	\end{align*}
	where $n:=[mK]+1$. Accordingly, we get for every $x\in\mathbb{R}^N$ and $\beta \in \mathbb{N}^N_0$ that 
	\begin{align*}
	\exp(-n\omega(x))|\partial^\beta (T\star f)(x)|&\leq \exp(n)\sum_{\alpha\in\mathbb{N}^N_0} \|\exp(-m\omega)f_\alpha\|_\infty \|\exp(n\omega)\partial^{\alpha+\beta} f\|_1\\&\leq \exp(n)\underset{\alpha\in\mathbb{N}^N_0}{\sup}\,\|\exp(-m\omega)f_\alpha\|_\infty\exp\left(m\varphi^*_\omega\left(\frac{|\alpha|}{m}\right)\right)\times\\&\quad\times\sum_{\alpha\in\mathbb{N}^N_0} \|\exp(n\omega)\partial^{\alpha+\beta} f\|_1\exp\left(-m\varphi^*_\omega\left(\frac{|\alpha|}{m}\right)\right).
	\end{align*}
Let $C:=\exp(n)\underset{\alpha\in\mathbb{N}^N_0}{\sup}\,\|\exp(-m\omega)f_\alpha\|_\infty\exp\left(m\varphi^*_\omega\left(\frac{|\alpha|}{m}\right)\right)<\infty$. Then it follows, 	thanks to \eqref{secondprop}, for every $k\geq m$, $x\in\mathbb{R}^N$ and $\beta \in \mathbb{N}^N_0$ that 
	\begin{align*}
	&\exp(-n\omega(x))|\partial^\beta (T\star f)(x)|\exp\left(-k\varphi^*_\omega\left(\frac{|\beta|}{k}\right)\right)\leq\\&\quad\leq C\sum_{\alpha\in\mathbb{N}^N_0} \|\exp(n\omega)\partial^{\alpha+\beta} f\|_1\exp\left(-m\varphi^*_\omega\left(\frac{|\alpha|}{m}\right)\right)\exp\left(-k\varphi^*_\omega\left(\frac{|\beta|}{k}\right)\right)\\&\quad\leq C\sum_{\alpha\in\mathbb{N}^N_0} \|\exp(n\omega)\partial^{\alpha+\beta} f\|_1\exp\left(-k\varphi^*_\omega\left(\frac{|\alpha|}{k}\right)\right)\exp\left(-k\varphi^*_\omega\left(\frac{|\beta|}{k}\right)\right)\\&\quad\leq C\sum_{\alpha\in\mathbb{N}^N_0} \|\exp(n\omega)\partial^{\alpha+\beta} f\|_1\exp\left(-2k\varphi^*_\omega\left(\frac{|\alpha+\beta|}{2k}\right)\right)=C\sigma_{2k,n,1}(f).
	\end{align*}
 So,
	\begin{equation*}
	r_{k,n}(T\star f)\leq C\sigma_{2k,n,1}(f).
	\end{equation*}
	If $k\leq m$, with a similar argument we obtain  for every $x\in\mathbb{R}^N$ and $\beta \in \mathbb{N}^N_0$ that
	\begin{align*}
	&\exp(-n\omega(x))|\partial^\beta (T\star f)(x)|\exp\left(-k\varphi^*_\omega\left(\frac{|\beta|}{k}\right)\right)\leq C\sigma_{2m,n,1}(f).
	\end{align*}
	This shows that $T\star f\in\bigcap_{k=1}^\infty \mathcal{O}^k_{n,\omega}(\mathbb{R}^N)\subset \mathcal{O}_{C,\omega}(\mathbb{R}^N)$ and that the map $f\mapsto T \star f$ is a continuous linear operator from $\mathcal{S}_\omega(\mathbb{R}^N)$ in $\bigcap_{k=1}^\infty \mathcal{O}^k_{n,\omega}(\mathbb{R}^N)$, and so in $\mathcal{O}_{C,\omega}(\mathbb{R}^N)$.
\end{proof}

We can now prove that $\mathcal{O}'_{C,\omega}(\mathbb{R}^N)$ is the space of convolutors of $\mathcal{S}_{\omega}(\mathbb{R}^N)$.

\begin{thm}\label{T.conv}
	Let $\omega$ be a non-quasianalytic weight function and $T\in \cS'_\omega(\R^N)$. 
	Consider the following properties.
	\begin{enumerate} 
		\item $T\in \cO'_{C,\omega}(\mathbb{R}^N)$. 
		\item For every $f\in \mathcal{S}_\omega(\mathbb{R}^N)$ we have $T\star f \in \mathcal{S}_\omega(\mathbb{R}^N)$.  \end{enumerate}
	Then $(1)\Rightarrow (2)$. If, in addition, the weight function $\omega$ satisfies the additional condition $\log (1+t)=\omega(t)$ as $t\to\infty$, then $(2)\Rightarrow (1)$.
	
	Moreover, if $T\in \cO'_{C,\omega}(\mathbb{R}^N)$, then the linear operator  $C_T\colon \mathcal{S}_\omega(\mathbb{R}^N)\to  \mathcal{S}_\omega(\mathbb{R}^N)$ defined by $C_T(f):=T\star f$, for $f\in \mathcal{S}_\omega(\mathbb{R}^N)$, is continuous.
\end{thm}
\begin{proof}
	$(1)\Rightarrow(2)$. Fix $f\in \cS_\omega(\R^N)$. Since $\mathcal{O}'_{C,\omega}(\mathbb{R}^N) \subset \mathcal{S}'_{\omega}(\mathbb{R}^N)$, the function $T\star f\in \mathcal{O}_{C,\omega}(\mathbb{R}^N)$ by Proposition \ref{conv}. We show that $T\star f\in \mathcal{S}_{\omega}(\mathbb{R}^N)$.

Fix any $p\in (1,\infty)$.	By Theorem \ref{struoc}, for every $n\in\mathbb{N}$ there exists $m\in\mathbb{N}$ such that
	\begin{equation}
	T=\sum_{\alpha\in\mathbb{N}^N_0} \partial^\alpha f_\alpha
	\end{equation}
	with $\{f_\alpha\}_{\alpha\in\N_0^N} \subset \left(\oplus L^{p}(\mathbb{R}^N,\exp(n\omega(x))\,dx)\right)_{\omega,m,p}$. So, for fixed $\beta\in\mathbb{N}^N_0$ and $x\in\mathbb{R}^N$, we have that
	\begin{align*}
	|\partial^\beta (T\star f)(x)| &\leq\sum_{\alpha\in\mathbb{N}^N_0} \int_{\mathbb{R}^N} |f_\alpha(y)\partial^{\alpha+\beta} f(x-y)|\, dy\\
	&\leq \sum_{\alpha\in\mathbb{N}^N_0} \|\exp(n\omega)f_\alpha\|_p \|\exp(-n\omega)\partial^{\alpha+\beta} \tau_x f\|_{p'},
	\end{align*}
where
	\begin{align*}
	\|\exp(-n\omega)\partial^{\alpha+\beta} \tau_x f\|_{p'}^{p'}&=\int_{\mathbb{R}^N} \exp(-np'\omega(y))|\partial^{\alpha+\beta} f(x-y)|\, dy\\&=\int_{\mathbb{R}^N} \exp(-np'\omega(x-z))|\partial^{\alpha+\beta} f(z)|\, dz \\&\leq \int_{\mathbb{R}^N} \exp\left(np'\left(1+\omega(z)-\frac{\omega(x)}{K}\right)\right)|\partial^{\alpha+\beta} f(z)|\, dz\\&\leq C^{p'}\exp\left(-\frac{np'\omega(x)}{K}\right)\|\exp(n\omega)\partial^{\alpha+\beta} f\|_{p'}^{p'},
	\end{align*}
	where $C:=\exp(n)$. Therefore, we get for every $x\in\mathbb{R}^N$ and $\beta \in \mathbb{N}^N_0$ that
	\begin{align*}
	\exp\left(\frac{n\omega(x)}{K}\right)|\partial^\beta (T\star f)(x)|&\leq C\sum_{\alpha\in\mathbb{N}^N_0} \|\exp(n\omega)f_\alpha\|_p \|\exp(n\omega)\partial^{\alpha+\beta} f\|_{p'}\\&\leq C\left(\sum_{\alpha\in\mathbb{N}^N_0} \|\exp(n\omega)f_\alpha\|_p^p\exp\left(pm\varphi^*_\omega\left(\frac{|\alpha|}{m}\right)\right)\right)^{\frac{1}{p}}\times\\
&\quad	\times \left(\sum_{\alpha\in\mathbb{N}^N_0} \|\exp(n\omega)\partial^{\alpha+\beta} f\|_{p'}^{p'}\exp\left(-p'm\varphi^*_\omega\left(\frac{|\alpha|}{m}\right)\right)\right)^{\frac{1}{p'}}\\
&=C'\left(\sum_{\alpha\in\mathbb{N}^N_0} \|\exp(n\omega)\partial^{\alpha+\beta} f\|_{p'}^{p'}\exp\left(-p'm\varphi^*_\omega\left(\frac{|\alpha|}{m}\right)\right)\right)^{\frac{1}{p'}},
	\end{align*}
	where $C':=C\left(\sum_{\alpha\in\mathbb{N}^N_0} \|\exp(n\omega)f_\alpha\|_p^p\exp\left(pm\varphi^*_\omega\left(\frac{|\alpha|}{m}\right)\right)\right)^{\frac{1}{p}}<\infty$. If $n>m$, it follows  for each $x\in\mathbb{R}^N$ and $\beta \in \mathbb{N}^N_0$ that
	\begin{align*}
	\exp\left(\frac{n\omega(x)}{K}\right)|\partial^\beta (T\star f)(x)|&\leq C'\left(\sum_{\alpha\in\mathbb{N}^N_0} \|\exp(n\omega)\partial^{\alpha+\beta} f\|_{p'}^{p'}\exp\left(-p'm\varphi^*_\omega\left(\frac{|\alpha|}{m}\right)\right)\right)^{\frac{1}{p'}} \\&\leq C'\left(\sum_{\alpha\in\mathbb{N}^N_0} \|\exp(n\omega)\partial^{\alpha+\beta} f\|_{p'}^{p'}\exp\left(-p'n\varphi^*_\omega\left(\frac{|\alpha|}{n}\right)\right)\right)^{\frac{1}{p'}}
	\end{align*}
	and so
	\begin{align*}
	&q_{\frac{n}{K},n}(T\star f)=\underset{\beta \in \mathbb{N}^N_0}{\sup}\,\underset{x \in \mathbb{R}^N}{\sup}\, \exp\left(-n\varphi^*_\omega\left(\frac{|\beta|}{n}\right)\right)\exp\left(\frac{n\omega(x)}{K}\right) |\partial^\beta (T\star f)(x)|\\&\quad\leq C'\left(\sum_{\alpha\in\mathbb{N}^N_0} \|\exp(n\omega)\partial^{\alpha+\beta} f\|_{p'}^{p'}\exp\left(-p'n\varphi^*_\omega\left(\frac{|\alpha|}{n}\right)\right)\exp\left(-p'n\varphi^*_\omega\left(\frac{|\beta|}{n}\right)\right)\right)^{\frac{1}{p'}}\\&\quad\leq C'\left(\sum_{\alpha\in\mathbb{N}^N_0} \|\exp(n\omega)\partial^{\alpha+\beta} f\|_{p'}^{p'}\exp\left(-2np'\varphi^*_\omega\left(\frac{|\alpha+\beta|}{2n}\right)\right)\right)^{\frac{1}{p'}}=\sigma_{2n,n,p'}(f).
	\end{align*}
	If $n\leq m$, procedding in a similar way  we obtain for every $x\in\mathbb{R}^N$ and $\beta \in \mathbb{N}^N_0$ that 
	\begin{align*}
	&q_{\frac{n}{K},n}(T\star f)\leq\sigma_{2m,n,p'}(f).
	\end{align*}
	From the arbitrarity of $n$, we conclude that $T\star f\in \mathcal{S}_{\omega}(\mathbb{R}^N)$. This shows also the continuity of the operator $C_T$.
	
	$(2)\Rightarrow(1)$. The assumption implies that   $T\star\phi \in \cS_\omega(\R^N)$ for every $\phi\in\cD_\omega(\R^N)$. For a fixed $p\in [1,\infty)$, we have by Remark \ref{R.Coincidenza} that  $\cS_{\omega}(\mathbb{R}^N)= \bigcap_{n=1}^\infty \cD_{L^p_n,\omega}(\mathbb{R}^N)$ and so,   $T\star\phi \in \cD_{L^p_n,\omega}(\mathbb{R}^N)$ for every $\phi\in\cD_\omega(\R^N)$ and $n\in\N$.  This yields by Remark \ref{inlp} that $T\star\phi \in L^p(\mathbb{R}^N,\exp(n\omega(x))dx)$ for every $\phi\in\cD_\omega(\R^N)$ and $n\in\N$. Therefore, by Theorem \ref{cardlpdual}$(2)\Rightarrow (3)$ and Proposition \ref{ocgd}, we get for every $n\in\N$  that there exist  an elliptic ultradifferentiable operator $G(D)$ of $\omega$-class and $f,g\in L^p(\mathbb{R}^N,\exp(n\omega(x))dx) $ such that $T=G(D)f+g$. So,  $T\in \cD'_{L^{p'}_{-n},\omega}(\R^N)$ for every $n\in\N$, i.e., $T\in \cO'_{C,\omega}(\R^N)$.
%
	 This completes the proof.
\end{proof}
 
Finally, we show that $\cO'_{C,\omega}(\R^N)$ is the space of convolutors of $\cS'_\omega(\R^N)$. To this end, we observe what follows.

\begin{defn}
	Let $\omega$ be a non-quasianalytic weight function. If  $T \in \cO'_{C,\omega}(\R^N)$ and $S\in \mathcal{S}'_\omega(\mathbb{R}^N)$, we define the convolution    $T\star S$ by
	\begin{equation}
	\langle  T\star S, f\rangle:= \langle S, \check{T}\star f\rangle
	\end{equation}
	for every $ f\in \mathcal{S}_\omega(\mathbb{R}^N)$ (recall that $\check{T}$ is the distribution defined by $\varphi\mapsto \langle \check{T},\varphi\rangle=\langle T,\check{\varphi}\rangle$).
\end{defn}
\begin{rem}\label{ocsw}
	The definition is well-placed. Indeed,  by Theorem \ref{T.conv} the convolution $\check{T}\star f$ belongs to $\cS_\omega(\mathbb{R}^N)$ whenever $f\in\cS_\omega(\R^N)$.   Furthermore,  the linear operator  $f\in\cS_\omega(\R^N)\mapsto\check{T}\star f\in\cS_\omega(\R^N)$ is continuous. So, as  $S$ is a continuous linear form on $\cS_\omega(\mathbb{R}^N)$, the map $f\mapsto \langle S, \check{T}\star f\rangle$ is continuous on $\mathcal{S}_\omega(\mathbb{R}^N)$, i.e., $T\star S \in \mathcal{S}'_\omega(\mathbb{R}^N)$.
\end{rem}

\begin{prop}Let $\omega$ be a non-quasianalytic weight function and  $T\in \cO'_{C,\omega}(\mathbb{R}^N)$.  Then 
		 $T\star S \in \mathcal{S}'_\omega(\mathbb{R}^N)$ for every $S\in \mathcal{S}'_\omega(\mathbb{R}^N)$.
	
	Moreover, the linear operator  $\cC_T\colon \mathcal{S}'_\omega(\mathbb{R}^N)\to  \mathcal{S}'_\omega(\mathbb{R}^N)$ defined by $\cC_T(S):=T\star S$, for $S\in \mathcal{S}'_\omega(\mathbb{R}^N)$, is continuous.
\end{prop}

\begin{proof} The proof of the first statement follows from Remark \ref{ocsw}.

We now observe that   for every $S\in \mathcal{S}'_\omega(\mathbb{R}^N)$ and $f\in \cS_\omega(\R^N)$ 
\begin{equation*}
		\langle T\star S, f\rangle =\langle S, \check{T}\star f\rangle.
\end{equation*}
This means that the linear operator $\cC_T\colon \cS'_\omega(\R^N)\to \cS'_\omega(\R^N)$ is the transpose of the continuous linear operator $C_T\colon \cS_\omega(\R^N)\to \cS_\omega(\R^N)$. Therefore, the linear operator  $\cC_T\colon \cS'_\omega(\R^N)\to \cS'_\omega(\R^N)$ is continuous too.
\end{proof}

\section{The action of the Fourier Transform in the multiplier and convolutor spaces}

In this final section we study the action of the Fourier transform on the spaces $\mathcal{O}_{M,\omega}(\mathbb{R}^N)$ and $\mathcal{O}'_{C,\omega}(\mathbb{R}^N)$ following the approach in \cite{Ki}. To this end, we recall that  the spaces $\cO_{M,\omega}(\R^N)$ and $\cO'_{C,\omega}(\R^N)$ are both continuously included in $\cS'_\omega(\R^N)$ (see 
 Theorem \ref{T.incl} and Proposition \ref{P.duali-Inc}).
We also recall that 
the gaussian function
$
f(x):=\exp\left(-\frac{|x|^2}{2}\right)
$, for $x\in\R^N$,
belongs to $\mathcal{S}_\omega(\mathbb{R}^N)$. 

\begin{thm}\label{omoc} Let $\omega$ be a non-quasianalytic weight function with the additional condition $\log(1+t)=o(\omega(t))$ as $t\to \infty$. Then 
	the Fourier transform $\mathcal{F}$ maps the space $\mathcal{O}'_{C,\omega}(\mathbb{R}^N)$ isomorphically onto the space $\mathcal{O}_{M,\omega}(\mathbb{R}^N)$. Furthermore, for $T\in \mathcal{O}'_{C,\omega}(\mathbb{R}^N)$ and $S\in \mathcal{S}'_{\omega}(\mathbb{R}^N)$, we have
	\begin{equation}\label{cm1}
	\mathcal{F}(T\star S)= \mathcal{F}(T)\mathcal{F}(S),
	\end{equation}
	and if $f \in \mathcal{O}_{M,\omega}(\mathbb{R}^N)$ and $T\in \mathcal{S}'_{\omega}(\mathbb{R}^N)$ we have
	\begin{equation}\label{cm2}
	\mathcal{F}(fT)= (2\pi)^{-N}\hat{f}\star\mathcal{F}(T).
	\end{equation}	
\end{thm}
\begin{proof}
	Fix $T\in \mathcal{O}'_{C,\omega}(\mathbb{R}^N)$. Since $\mathcal{O}'_{C,\omega}(\mathbb{R}^N)\subset \mathcal{S}'_{\omega}(\mathbb{R}^N)$ and $\mathcal{F}$ is an automorphism into $\mathcal{S}'_{\omega}(\mathbb{R}^N)$, we have that $\mathcal{F}(T)\in\mathcal{S}'_{\omega}(\mathbb{R}^N)$. So, we have to prove that $\mathcal{F}(T)$ is an ultradistribution represented by a function belonging to $\mathcal{O}_{M,\omega}(\mathbb{R}^N)$. To do this, we observe that  by (\ref{trf})   the ultradistribution $\exp\left(-\frac{|x|^2}{2}\right)\mathcal{F}(T)\in \mathcal{S}'_{\omega}(\mathbb{R}^N)$ as the gaussian function belongs to $\mathcal{S}_{\omega}(\mathbb{R}^N)$, and it is equal to the ultradistribution $(2\pi)^{\frac{N}{2}}\mathcal{F}\left(T\star\exp\left(-\frac{|x|^2}{2}\right)\right)$. On the other hand,  from Theorem \ref{T.conv} it follows that $T\star \exp\left(-\frac{|x|^2}{2}\right) \in \mathcal{S}_{\omega}(\mathbb{R}^N)$. Since $\mathcal{F}$ is an automorphism into $\mathcal{S}_{\omega}(\mathbb{R}^N)$, then $(2\pi)^{\frac{N}{2}}\mathcal{F}\left(T\star\exp\left(-\frac{|x|^2}{2}\right)\right)\in\mathcal{S}_{\omega}(\mathbb{R}^N)$.Thus, the ultradistribution $\mathcal{F}(T)$ is represented by the function $\psi:=\exp\left(\frac{|x|^2}{2}\right)(2\pi)^{\frac{N}{2}}\mathcal{F}\left(T\star\exp\left(-\frac{|x|^2}{2}\right)\right)\in C^\infty(\mathbb{R}^N)$. But,  $\psi \in \mathcal{E}_{\omega}(\mathbb{R}^N)$ as it is easy to see. So, it remains to prove that the function $\psi$ belongs to $\mathcal{O}_{M,\omega}(\mathbb{R}^N)$, i.e.,  thanks to Theorem \ref{T.Multiplier} it suffices to show   that $f\psi \in \mathcal{S}_{\omega}(\mathbb{R}^N)$ for every $f\in\mathcal{S}_{\omega}(\mathbb{R}^N)$. To see this, we fix $f\in\mathcal{S}_{\omega}(\mathbb{R}^N)$ and notice that by formula (\ref{trf})
	\begin{equation*}
	f\psi=f\mathcal{F}(T)=\mathcal{F}(T\star \mathcal{F}^{-1}(f))\in\mathcal{S}_{\omega}(\mathbb{R}^N).
	\end{equation*}
	This shows that $\mathcal{F}(T)\in\cO_{M,\omega}(\mathbb{R}^N)$.

	In order to get that $\mathcal{F}$ is an isomorphism between $\mathcal{O}'_{C,\omega}(\mathbb{R}^N)$ and $\mathcal{O}_{M,\omega}(\mathbb{R}^N)$, it remains to prove that if $f\in\mathcal{O}_{M,\omega}(\mathbb{R}^N)$, then $\hat{f}\in\mathcal{O}'_{C,\omega}(\mathbb{R}^N)$. So, fixed $f\in\mathcal{O}_{M,\omega}(\mathbb{R}^N)$, we observe that $\mathcal{F}\hat{f} \in (\mathcal{F} \circ \mathcal{F}) (\mathcal{O}_{M,\omega}(\mathbb{R}^N))= (2\pi)^{N}(\mathcal{O}_{M,\omega}(\mathbb{R}^N))^\vee=\mathcal{O}_{M,\omega}(\mathbb{R}^N)$. By  Theorem \ref{T.Multiplier} it follows  for every $g\in \mathcal{S}_{\omega}(\mathbb{R}^N)$ that  $(\mathcal{F}\hat{f})\hat{g}\in \mathcal{S}_{\omega}(\mathbb{R}^N)$. But, taking into account that  $\hat{f}\in \cS'_\omega(\R^N)$ and that 
	\begin{align*}
	(\mathcal{F}\hat{f})\hat{g}\in \mathcal{S}_{\omega}(\mathbb{R}^N)\quad \forall g\in \mathcal{S}_{\omega}(\mathbb{R}^N) &\iff \mathcal{F}(\hat{f}\star g) \in \mathcal{S}_{\omega}(\mathbb{R}^N)\quad \forall g\in \mathcal{S}_{\omega}(\mathbb{R}^N)\\& \iff \hat{f}\star g \in \mathcal{S}_{\omega}(\mathbb{R}^N)\quad \forall g\in \mathcal{S}_{\omega}(\mathbb{R}^N),
	\end{align*}
  we obtain via  Theorem \ref{T.conv} that $\hat{f}\in \cO'_{C,\omega}(\R^N)$.

	We now prove formula $(\ref{cm1})$. Fixed $T\in \mathcal{O}'_{C,\omega}(\mathbb{R}^N)$, $S\in \mathcal{S}'_{\omega}(\mathbb{R}^N)$, $f\in \mathcal{S}_\omega(\mathbb{R}^N)$, we have
	\begin{align*}
		\langle \mathcal{F}(T\star S), f\rangle= \langle T\star S, \hat{f} \rangle=\langle S, \check{T}\star \hat{f}  \rangle, 
	\end{align*}
where $\hat{f}\in \mathcal{S}_\omega(\mathbb{R}^N)$ and  $\check{T}\star \hat{f}\in\cS_\omega(\R^N)$  by Theorem  \ref{T.conv}, and hence,  $\check{T}\star \hat{f}\in\mathcal{S}'_{\omega}(\mathbb{R}^N)$. Thanks to formula $(\ref{trf})$, we  obtain $\cF(\cF(T)f)=\check{T}\star\hat{f}$. So, we conclude
	\begin{align*}
		\langle \mathcal{F}(T\star S), f\rangle= \langle S, \cF(\cF(T)f)  \rangle= \langle \cF (S), \cF (T)f  \rangle=\langle \mathcal{F}(T)\cF (S), f  \rangle,
	\end{align*}
	where in the last equality we used the property that $\cF (T)\in \cO_{M,\omega}(\R^N)$.
	This shows that $(\ref{cm1})$ holds. The proof of formula $(\ref{cm2})$ is analogous and so it is omitted. 
\end{proof}

Let $X$ be a Hausdorff lc-space, $\Gamma_X$  a system of continuous seminorms generating the topology of $X$ and $\cB(X)$ denote the  collection of all bounded subsets of $X$. 
We recall that the topology $\tau_b$ of uniform convergence on bounded sets in the space $\cL(X)$ of all continuous linear operators from $X$ into itself is defined  by the seminorms $q_B(S):=\sup_{x\in B}q(Sx)$ ($S\in \cL(X)$) for each $B\in \cB(X)$ and $q\in \Gamma_X$. In such a case we write $\cL_b(X)$.

By Theorem \ref{T.Multiplier} the space $\cO_{M,\omega}(\R^N)$ can be identified with the space $\cM(\cS_{\omega}(\R^N))$ of all multipliers on $\cS_\omega(\R^N)$ via the map $M\colon \cO_{M,\omega}(\R^N)\to \cM(\cS_{\omega}(\R^N))$ defined by $M(f):=M_f$ for each $f\in \cO_{M,\omega}(\R^N)$. Since  $\cM(\cS_{\omega}(\R^N))$ is a subspace of $\cL(\cS_{\omega}(\R^N))$, the space $\cO_{M,\omega}(\R^N)$ (via the map $M$) can be then endowed with the topology $\tau_b$ induced by $\cL_b(\cS_{\omega}(\R^N))$. In view of \cite[Theorem 5.2]{De} we have that $t=\tau_b|_{\cO_{M,\omega}(\R^N)}$. Therefore, a net $\{f_i\}_{i\in I}\subset \cO_{M,\omega}(\R^N)$ $t$-converges to $f\in \cO_{M,\omega}(\R^N)$ if and only if $\sup_{g\in B}q_{m,n}((f_i-f)g)\to 0$ for every $m,n\in\N$ and $B\in \cB(\cS_\omega(\R^N))$. 

On the other hand, by Theorem \ref{T.conv} the space $\cO'_{C,\omega}(\R^N)$ can be  identified with the space $\cC(\cS_{\omega}(\R^N))$ of all convolutors on $\cS_\omega(\R^N)$ via the map $C\colon \cO'_{C,\omega}(\R^N)\to \cC(\cS_{\omega}(\R^N))$ defined by $C(T):=C_T$ for each $T\in \cO'_{C,\omega}(\R^N)$. Since  $\cC(\cS_{\omega}(\R^N))$ is a subspace of $\cL(\cS_{\omega}(\R^N))$, the space $\cO_{C,\omega}(\R^N)$ (via the map $C$) can be then endowed with the topology $\tau_b$ induced by $\cL_b(\cS_{\omega}(\R^N))$. Therefore, a net $\{T_i\}_{i\in I}\subset \cO'_{C,\omega}(\R^N)$ $\tau_b$-converges to $T\in \cO'_{C,\omega}(\R^N)$ if and only if $\sup_{g\in B}q_{m,n}((T_i-T)\star g )\to 0$ for every $m,n\in\N$ and $B\in \cB(\cS_\omega(\R^N))$.  Now, we can prove this continuity result.

\begin{thm}
Let $\omega$ be a non-quasianalytic weight function with the additional condition $\log(1+t)=o(\omega(t))$ as $t\to\infty$. Then the Fourier transform $\cF$ is a topological isomorphism from $(\cO'_{C,\omega}(\R^N), \tau_b)$ onto $\cO_{M,\omega}(\R^N)$.
\end{thm}

\begin{proof} Let $\{T_i\}_{i\in I}\subset \cO'_{C,\omega}(\R^N)$ be a  $\tau_b$-convergent net in  $\cO'_{C,\omega}(\R^N)$ with limit equal to $T\in  \cO'_{C,\omega}(\R^N)$.  Then  $\cF(T_i\star g)\to \cF(T\star g)$ in $\cS_\omega(\R^N)$ uniformly on bounded subsets of $\cS_\omega(\R^N)$. On the other hand,  by Theorem \ref{omoc} the functions $f_i:=\cF(T_i)$, for $i\in I$, and $f:=\cF(T)$, belong to  $\cO_{M,\omega}(\R^N)$.  Moreover, for every $i\in I$, by \eqref{trf} we have for every $g\in \cS_\omega(\R^N)$ that 
	\begin{align}\label{limit}
	&\cF|_{\cS_\omega(\R^N)}(T_i\star g)=\cF(T_i\star g)=\hat{g}\cF(T_i), \nonumber \\ &\cF|_{\cS_\omega(\R^N)}(T\star g)=\cF(T\star g)=\hat{g}\cF(T),
	\end{align}
	where the restriction $\cF|_{\cS_\omega(\R^N)}$ of $\cF\colon \cS'_\omega(\R^N)\to \cS'_\omega(\R^N)$ coincides with the Fourier transfom on $\cS_\omega(\R^N)$, which is a topological isomorphism on $\cS_\omega(\R^N)$. So, $\cF(T_i\star g)\to \cF(T\star g)$ in $\cS_\omega(\R^N)$ uniformly on bounded subsets of $\cS_\omega(\R^N)$ if and only if  $\cF|_{\cS_\omega(\R^N)}(T_i\star g)\to \cF|_{\cS_\omega(\R^N)}(T\star g)$ in $\cS_\omega(\R^N)$. We now observe that  for a subset $B$ of $\cS_\omega(\R^N)$, the set  $\{\hat{g}\colon g\in B\}\in \cB(\cS_\omega(\R^N))$ if and only if $B\in \cB(\cS_\omega(\R^N))$. In view of the equalities in \eqref{limit}, it follows that $f_i\to f$ in $\cO_{M,\omega}(\R^N)$. In a similar way one shows that if $f_i\to f$ in $\cO_{M,\omega}(\R^N)$, then $\tau_b$-$\lim T_i=T$.
	\end{proof}

\end{document}